\documentclass{amsart}   

\usepackage{amssymb}
\usepackage{amsmath}
\usepackage{amsthm}
\usepackage{enumerate,enumitem}
\usepackage[dvipsnames]{xcolor}
\usepackage[urlcolor=black, colorlinks=true, citecolor=black, linkcolor=black]{hyperref}

\usepackage{wrapfig}

\usepackage{tikz}
\usetikzlibrary{matrix,arrows}

\newcommand{\semicircle}[3]{\draw[thick] (#1+#3,#2) arc [start angle=0, end angle=180, radius=#3];}

\newcommand{\ellipse}[3]{\draw  (#1,#2) ellipse (#3 and #3/2);}

\newcommand{\script}{\mathcal}
\newcommand{\parentheses}[1]{{\left( {#1} \right)}}

\newcommand{\p}{\parentheses}
\newcommand{\of}{\parentheses}
\newcommand{\closure}[1]{\overline{#1}}

\newcommand{\bd}[1]{\mathrm{Bd}\of{#1}}
\newcommand{\Set}[1]{{\left\lbrace {#1} \right\rbrace}}
\newcommand{\singleton}{\Set}

\newcommand{\cardinality}[1]{{\left\lvert {#1} \right\rvert}}

\def\set#1:#2{\Set{{#1} \colon {#2}}}

\newcommand{\diam}[1]{\textnormal{diam}{\left({#1} \right)}}

\theoremstyle{plain}
\newtheorem{theorem}{Theorem}
\newtheorem*{theorem*}{Theorem}
\newtheorem{lemma}[theorem]{Lemma}
\newtheorem{cor}[theorem]{Corollary}
\newtheorem{prop}[theorem]{Proposition}

\theoremstyle{definition}
\newtheorem{defn}[theorem]{Definition}

\newcommand{\N}{\mathbb{N}}
\newcommand{\Q}{\mathbb{Q}}

\title[Graph-Like Compacta: Characterizations and Eulerian Loops]{Graph-Like Compacta: \\ Characterizations and Eulerian Loops}

\author{Benjamin Espinoza}
\address{Department of Mathematics, University of Pittsburgh at Greensburg, Greensburg, PA~15601, USA}
\email{bee1@pitt.edu}

\author{Paul Gartside}
\address{Department
    of Mathematics, University of Pittsburgh, Pittsburgh, PA~15260, USA}
\email{gartside@math.pitt.edu} 

\author{Max Pitz}
\address{Department of Mathematics, Bundesstra\ss e 55, 20146 Hamburg, Germany}
\email{max.pitz@uni-hamburg.de}



\begin{document}

\begin{abstract}
A \emph{compact graph-like} space is a triple $(X,V,E)$ where $X$ is a compact, metrizable space,  $V \subseteq X$ is a closed zero-dimensional  subset, and  $E$ is an index set such that $X \setminus V \cong E \times (0,1)$. New characterizations of compact graph-like spaces are given, connecting them to certain classes of continua, and to standard subspaces of Freudenthal compactifications of locally finite graphs. These are applied to  characterize  Eulerian graph-like compacta.
\end{abstract}

\maketitle

\section{Introduction}

Locally finite graphs can be compactified, to form the Freudenthal compactification, by adding their ends. This topological setting provides what appears to be the `right' framework for studying locally finite graphs. Indeed, many classical theorems from finite graph theory that involve paths or cycles have been shown to generalize to locally finite infinite graphs in this topological setting, while failing to extend in a purely graph theoretic setting. See the survey series \cite{DSurv}. 
More recently, compact graph-like spaces were introduced by Thomassen and Vella,  \cite{thomassenvella}, as a natural class  encompassing graphs, and in particular containing the standard subspaces of Freudenthal compactification of locally finite graphs.

A \emph{compact graph-like} space is a triple $(X,V,E)$ where: $X$ is a compact, metrizable space,  $V \subseteq X$ is a closed zero-dimensional  subset, and  $E$ is a discrete index set such that $X \setminus V \cong E \times (0,1)$. 
The sets $V$ and $E$ are the \emph{vertices} and \emph{edges} of $X$ respectively. 
More generally,  a topological space $X$ is compact graph-like, if there exists $V \subseteq X$ and a  set $E$ such that $(X,V,E)$ is a compact graph-like space. Recall that connected compact metrizable spaces are called \emph{continua}, and so a 
\emph{graph-like continuum} is a continuum which is graph-like.

Papers in which  graph-like spaces have played a key role include \cite{thomassenvella} where several Menger-like results are given, and  \cite{graphlikeplanar} where  algebraic criteria for the planarity of graph-like continua are presented. In \cite{infinitematroids}, aspects of the matroid theory for graphs have been generalized to infinite matroids on graph-like spaces. 

In this paper we present two groups of new results. The first group consists of characterizations 
of compact graph-like spaces and continua. These connect graph-like continua to certain classes of continua which have been intensively studied by continua theorists. We also establish that compact graph-like spaces are not simply `like' the Freudenthal compactifications of locally finite graphs, but in fact \emph{are} standard subspaces of the latter. 
Our second group of results consists of various characterizations of when a graph-like continuum is Eulerian. These naturally extend  classical results for graphs.

\subsection{The Main Theorems}

In Section~\ref{sec:repchar} we give various characterizations and representations of compact graph-like spaces, and graph-like continua, which demonstrate that graph-like continua form a class of continua which are also of considerable interest from the point of view of continua theory. These results can be summarized as follows.
\begin{theorem*}[A]
The following are equivalent for a continuum $X$:
\begin{enumerate}[label=(\roman*)]
	\item $X$ is graph-like,    
    \item $X$ is regular and has a closed zero-dimensional subset $V$ such that all points outside of $V$ have order $2$,
    \item $X$ is completely regular,
    \item[(iv)]  $X$ is a countable inverse limit of finite connected multi-graphs with onto, monotone, simplicial bonding maps with non-trivial fibres at vertices only,
    \item[$(iv)'$]  $X$ is a countable inverse limit of finite connected multi-graphs with onto, monotone bonding maps that project vertices onto vertices, and 
    \item[(v)] $X$ is homeomorphic to a connected standard subspace of a Freudenthal compactification of a locally finite graph.
\end{enumerate}
\end{theorem*}

Here a continuum is \emph{regular} if it has a base all of whose members have finite boundary, and \emph{completely regular} if all non-trivial subcontinua have non-empty interior. A map is \emph{monotone} if all fibres are connected, while a map between graph-like spaces is \emph{simplicial} if it maps vertices to vertices, and edges either homeomorphically to another edge, or to a vertex. 
A \emph{standard subspace}  of a compact graph-like space is a closed subspace that contains all edges it intersects.
The equivalence of (i) and (ii) is analogous to a well-known topological characterization of finite graphs, namely a continuum is a graph if and only if every point has finite order, and all but finitely many points have order $2$, \cite[Theorem 9.10 \& 9.13]{Nadler}.
The equivalence of (i) and (iv) provides a powerful tool to lift results in finite graph theory to graph-like continua. Indeed this is key to our results on Eulerian paths and loops below. It also is key to the equivalence of (i) and (v). 
The equivalence of (i) and (iii) yields a purely internal topological characterization of graph-like continua, without any reference to distinguished points, `vertices', or subsets, `edges'.

We prove all of Theorem~(A) taking `compact graph-like space' as the basic notion. Because `compact graph-like' takes a middle ground between topology and graph theory, our proofs are clean and efficient. However it is important to note that the equivalence of (i) and (iii) follows, modulo some basic lemmas, from a result of Krasinkiewicz, \cite{kra}, while the implication (iii) implies (iv) is essentially shown by Nikiel in \cite{inverselimitsnikiel}. Nikiel also claimed,  without proof, the converse implication. Regarding (v), Bowler et al.\ have claimed, without proof, the weaker assertion that every compact graph-like space is a minor (essentially: a quotient) of a Freudenthal  compactification of some locally finite graph, \cite[p.~6]{infinitematroids}.

\medskip

In Sections~\ref{EulerPf} and~\ref{sec:BruhnStein} 
we extend some well-known characterizations of Eulerian graphs to graph-like continua. 
Let $G$ be a (multi-)graph. A \emph{trail} in $G$ is an edge path with no repeated edges. It is \emph{open} if the start and end vertices are distinct, and \emph{closed} if they coincide. We also call closed trails \emph{circuits}. A \emph{segment} is a trail which does not cross itself.  A \emph{cycle} is a circuit which never crosses itself. A trail is \emph{Eulerian} if it contains all edges of the graph. (Note that an \emph{Eulerian circuit} is a closed Eulerian trail.) The graph $G$ is \emph{Eulerian} (respectively, \emph{closed Eulerian}) if it has an open (respectively, closed) Eulerian trail; and \emph{Eulerian} if it either open or closed Eulerian. Call a vertex $v$ of a graph $G$ \emph{even} (respectively, \emph{odd} if  the degree of $v$ in $G$ is even (respectively, odd).

Classical results of Euler and Veblen  characterize  multi-graphs  with closed, and  respectively, open, Eulerian trails as follows.
Let $G$ be a connected multi-graph with vertex set $V$, then the following are equivalent: (i) $G$ is closed [open] Eulerian,
(ii) every vertex is even [apart from precisely two vertices which are odd],
(iii) {[there are vertices $x \neq y$ such that]} for every bi-partition of $V$, the number of cross edges is even [if and only if $x$ and $y$ lie in the same part], and (iv) the edges of $G$ can be partitioned into edge disjoint cycles [and a non-trivial segment].
We extend these results to compact  graph-like spaces, and  prove the following result.
\begin{theorem*}[B]
Let $X$ be a graph-like continuum with vertices $V$. The following are equivalent:
\begin{enumerate}
\item[(i)] $X$ is closed [open] Eulerian,
\item[(ii)] every vertex is even [apart from precisely two vertices which are odd],
\item[(ii)${}'$] every vertex has strongly even degree [apart from precisely two vertices which have strongly odd degree],
\item[(iii)] {[there are vertices $x \neq y$ such that]} for every partition of $V$ into two clopen pieces, the number of cross edges is even [if and only if $x$ and $y$ lie in the same part], and
\item[(iv)] the edges of $X$ can be partitioned into edge-disjoint circles [and a non-trivial arc].
\end{enumerate}

Further, if $X$ is closed [open] Eulerian then either $X$ has continuum many distinct Eulerian loops [Eulerian paths], or has a finite  number of distinct Eulerian loops [Eulerian paths], which occurs if and only if $X$ is homeomorphic to a finite closed [open] Eulerian graph.
\end{theorem*}

Let $X$ be a compact graph-like space with set of vertices $V$.   A subspace of $X$ is called an \emph{arc} if it is homeomorphic to $I=[0,1]$, and a \emph{circle} if it is homeomorphic to the circle, $S^1$.  A (standard) \emph{path} is a continuous map $f\colon I \to X$ such that $f(0)$ and $f(1)$ are vertices, $f$ is injective on the interior of every edge and $f^{-1} (V)$ has empty interior. Note that every continuous map $f\colon I \to X$ with $f(0), f(1)$ as vertices is homotopy equivalent (with fixed endpoints) to a path. Also note that if $X$ is a graph (with usual topology), then every path yields a corresponding trail, and every trail corresponds to a path. A path, $f$, is \emph{open} if $f(0) \ne f(1)$, and \emph{closed} if $f(0)=f(1)$. Closed paths are called \emph{loops}. A path (or loop) is \emph{Eulerian} if its image contains every edge. Note that in a graph with the usual topology there is a natural correspondence between Eulerian paths and Eulerian trails, and Eulerian circuits and Eulerian loops.
We abbreviate `closed and open' to `clopen'. A vertex $v$ is  \emph{odd} (resp.\@ \emph{even}) if and only if there exists a clopen subset $A$ of $V$ containing $v$, such that for every clopen subset $C$ of the vertices $V$ with $v \in C \subseteq A$ the number of edges between $C$ and $V \setminus C$ is odd (resp.\ even).

The equivalence of (i), (ii), (iii) and (iv) in Theorem~(B) is established in Section~\ref{pfthmab}. At the heart of our proof is our representation of compact graph-like spaces as inverse limits, and an induced inverse limit representation of all Eulerian loops (possibly empty, of course). The `further' part of Theorem~(B) follows in Section~\ref{eulercnt} from topological considerations of the space of all Eulerian loops.

Our definition of `even' and `odd' vertices is natural within the context of Theorem~(B). An alternative approach to degree, due to Bruhn and Stein \cite{euler}, leads to the notions of `strongly even degree' and `strongly odd degree' appearing in item (ii)${}'$ of Theorem~(B). See Section~\ref{sec:BruhnStein} for details and the proof that `(ii) implies (ii)${}'$' and `(ii)${}'$ implies (i)'.

\subsubsection*{Alternative Paths} Theorem~(A) shines an unexpected light on connections between concepts from continua theory (completely regular continua and inverse limits of graphs), and concepts arising from infinite graph theory (graph-like continua, Freudenthal compactifications of graphs, and their standard subspaces). As a result we discover that the various parts of Theorem~(B) generalize numerous results in the literature, and--with the considerable assistance of the machinery developed here--Theorem~(B) can be derived from older work.

For Freudenthal compactification of graphs, the equivalence of (i), (iii) and (iv) is due to Diestel \& K\"{u}hn, \cite[Theorem 7.2]{infinitecycles}, while the equivalence with (ii)${}'$ is  due to Bruhn \& Stein, \cite[Theorem 4]{euler}. For standard subspaces of Freudenthal compactification of graphs,
    the equivalence of (iii) and (iv) is due to Diestel \& K\"{u}hn, \cite[Theorem 5.2]{infinitecycles2}, 
    the equivalence of (i) and (iv) is due to Georgakopoulos, \cite[Theorem 1.3]{agelos}, and the 
    equivalence (ii)${}'$ and (iii) is due to Berger \& Bruhn, \cite[Theorem 5]{standard}. 
It should also be noted that the method of lifting Eulerian paths and loops via inverse limits, used by Georgakopolous, was previously introduced by Bula et al.\ in \cite[Theorem 5]{Koenigsberg}.

Thus an alternative path to proving Theorem~(B) is as follows. Let $X$ be a graph-like continuum. According to the equivalence of (i) and (v) in Theorem~(A), which depends on the equivalence of (i) and (iv), $X$ is homeomorphic to a standard subspace of a Freudenthal compactification of a graph. Now equivalence of (i), (ii)${}'$, (iii) and (iv) follows from the results cited immediately above. To add equivalence of (ii) apply Theorem~\ref{thm_menger} (which uses Theorem~(A) (i)$\iff$(iv), and a non-trivial inverse limit argument) and Lemmas~\ref{lem_even} and~\ref{lem_deg}. 
Although this alternative path exists, the direct proofs given here in Sections~\ref{EulerPf} and~\ref{sec:BruhnStein}, using compact graph-like spaces as the basic notion, are--in the authors' view--much shorter and more natural.

\subsection{Examples}
Before proving our results on  graph-like continua, we now introduce some examples. 
With these examples we have three objectives. First show a little of the variety of graph-like continua. Second elucidate some of the less familiar terms in Theorem~(B), in particular `even' and `odd' vertices. Third  demonstrate the remarkable complexity of Eulerian loops and paths in graph-like continua. 
This complexity highlights the hidden depths of Theorem~(B).

\subsubsection*{Example 1.} The two-way infinite ladder with single diagonals, which  is the infinite graph $G$ shown below. Notice that all its vertices are even, but it has no Eulerian loop. 

\medskip

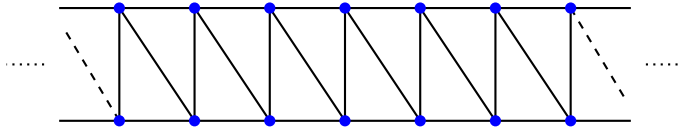
\begin{center}
\begin{tikzpicture}[thick]

\draw (-3.8,0.5) -- (3.8,0.5);
\draw (-3.8,-1) -- (3.8,-1);

\draw (0,0.5) -- (0,-1);
\draw (0,-1) -- (-1,0.5);

\draw (1,0.5) -- (1,-1);
\draw (1,-1) -- (0,0.5);

\draw (2,0.5) -- (2,-1);
\draw (2,-1) -- (1,0.5);

\draw (3,0.5) -- (3,-1);
\draw (3,-1) -- (2,0.5);

\draw[dashed] (3,0.5) -- (3.75,-0.75);

\draw (-1,0.5) -- (-1,-1);
\draw (-1,-1) -- (-2,0.5);

\draw (-2,0.5) -- (-2,-1);
\draw (-2,-1) -- (-3,0.5);

\draw (-3,0.5) -- (-3,-1);

\draw[dashed] (-3,-1) -- (-3.75,0.25);

\draw[dotted] (-4,-0.25) -- (-4.5,-0.25);
\draw[dotted] (4,-0.25) -- (4.5,-0.25);

\draw (0,0.5) node[circle, inner sep=1.5, fill=blue] {};
\draw (0,-1) node[circle, inner sep=1.5, fill=blue] {};

\draw (1,0.5) node[circle, inner sep=1.5, fill=blue] {};
\draw (1,-1) node[circle, inner sep=1.5, fill=blue] {};

\draw (-1,0.5) node[circle, inner sep=1.5, fill=blue] {};
\draw (-1,-1) node[circle, inner sep=1.5, fill=blue] {};

\draw (2,0.5) node[circle, inner sep=1.5, fill=blue] {};
\draw (2,-1) node[circle, inner sep=1.5, fill=blue] {};

\draw (-2,0.5) node[circle, inner sep=1.5, fill=blue] {};
\draw (-2,-1) node[circle, inner sep=1.5, fill=blue] {};

\draw (3,0.5) node[circle, inner sep=1.5, fill=blue] {};
\draw (3,-1) node[circle, inner sep=1.5, fill=blue] {};

\draw (-3,0.5) node[circle, inner sep=1.5, fill=blue] {};
\draw (-3,-1) node[circle, inner sep=1.5, fill=blue] {};

\end{tikzpicture}
\end{center}

\medskip

The Freudenthal compactification, $\gamma G$, of $G$ adds two ends. Then, as shown in the diagram, $\gamma G$ has an open Eulerian path from one end to the other.

\bigskip

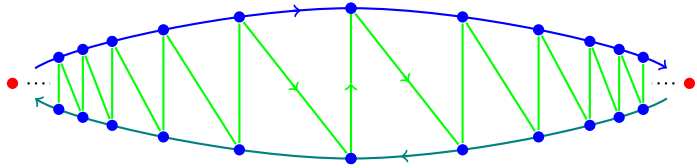
\begin{center}
\begin{tikzpicture}[thick]

\draw (-4.5,0) node[circle, inner sep=1.5, fill=red] {};
\draw (4.5,0) node[circle, inner sep=1.5, fill=red] {};

\draw[dotted] (-4.3,0) -- (-4,0);
\draw[dotted] (4.3,0) -- (4,0);

\draw[green,->] (0,-1) -- (0,0);
\draw[green] (0,0) -- (0,1);

\draw[blue] (-4.2,0.2) .. controls (-3.75,0.5) and (-1.5,1) .. (0,1)
node[circle,pos=0.15,fill=blue,inner sep=1.5] (p1) {}
node[circle,pos=0.25,fill=blue,inner sep=1.5] (p2) {}
node[circle,pos=0.35,fill=blue,inner sep=1.5] (p3) {}
node[circle,pos=0.5,fill=blue,inner sep=1.5] (p4) {}
node[circle,pos=0.7,fill=blue,inner sep=1.5] (p5) {} 
node[pos=0.85,->,sloped] {\tikz \draw[->,thick] (0,0) -- ++(0.1,0);}

;

\draw[<-,green!50!blue] (-4.2,-0.2) .. controls (-3.75,-0.5) and (-1.5,-1) .. (0,-1)
node[circle,pos=0.15,fill=blue,inner sep=1.5] (n1) {}
node[circle,pos=0.25,fill=blue,inner sep=1.5] (n2) {}
node[circle,pos=0.35,fill=blue,inner sep=1.5] (n3) {}
node[circle,pos=0.5,fill=blue,inner sep=1.5] (n4) {}
node[circle,pos=0.7,fill=blue,inner sep=1.5] (n5) {}
;

\draw (0,1) node[circle, inner sep=1.5, fill=blue] (p6) {};
\draw (0,-1) node[circle, inner sep=1.5, fill=blue] (n6) {};

\draw[green] (p1) -- (n1);
\draw[green] (p1) -- (n2);

\draw[green] (p2) -- (n2);
\draw[green] (p2) -- (n3);

\draw[green] (p3) -- (n3);
\draw[green] (p3) -- (n4);

\draw[green] (p4) -- (n4);
\draw[green] (p4) -- (n5);

\draw[green] (p5) -- (n5);
\draw[green] (p5) -- (n6)
node[midway,->,sloped] {\tikz \draw[->,thick] (0,0) -- ++(0.1,0);}
;

\draw[<-,blue] (4.2,0.2) .. controls (3.75,0.5) and (1.5,1) .. (0,1)
node[circle,pos=0.15,fill=blue,inner sep=1.5] (p1) {}
node[circle,pos=0.25,fill=blue,inner sep=1.5] (p2) {}
node[circle,pos=0.35,fill=blue,inner sep=1.5] (p3) {}
node[circle,pos=0.5,fill=blue,inner sep=1.5] (p4) {}
node[circle,pos=0.7,fill=blue,inner sep=1.5] (p5) {}
;

\draw[blue!50!green] (4.2,-0.2) .. controls (3.75,-0.5) and (1.5,-1) .. (0,-1)
node[circle,pos=0.15,fill=blue,inner sep=1.5] (n1) {}
node[circle,pos=0.25,fill=blue,inner sep=1.5] (n2) {}
node[circle,pos=0.35,fill=blue,inner sep=1.5] (n3) {}
node[circle,pos=0.5,fill=blue,inner sep=1.5] (n4) {}
node[circle,pos=0.7,fill=blue,inner sep=1.5] (n5) {}
node[pos=0.85,->,sloped] {\tikz \draw[<-,thick] (0,0) -- ++(0.1,0);}

;

\draw (0,1) node[circle, inner sep=1.5, fill=blue] (p6) {};
\draw (0,-1) node[circle, inner sep=1.5, fill=blue] (n6) {};

\draw[green] (p1) -- (n1);
\draw[green] (n1) -- (p2);

\draw[green] (p2) -- (n2);
\draw[green] (n2) -- (p3);

\draw[green] (p3) -- (n3);
\draw[green] (n3) -- (p4);

\draw[green] (p4) -- (n4);
\draw[green] (n4) -- (p5);

\draw[green] (p5) -- (n5);
\draw[green] (p6) -- (n5)
node[midway,->,sloped] {\tikz \draw[->,thick] (0,0) -- ++(0.1,0);}
;

\end{tikzpicture}
\end{center}

It follows from Theorem~(B) that the two ends are odd. We now demonstrate that the left end is odd from the definition. 
For the clopen neighborhood $A$ take the left end along with all vertices to the left of some rung of the ladder. Now consider an arbitrary  clopen $C$ containing the left end and contained in $A$ (depicted by the green vertices in the diagram below). 
Then $C$ is the disjoint union of a $C_0$ which contains the end and all vertices to the left of a rung, and a $C_1$ which is a finite subset of $A \setminus C_0$. 
In the diagram we see that the number of edges from $C$ to $V \setminus C$ is $9$, which is odd. 
In general, if we identify $C_0$ (and all edges between members of $C_0$) to a vertex $v$ and identify $V \setminus A$ to a vertex $w$, then we get a finite graph with exactly two odd vertices (namely, $v$ and $w$, both of degree $3$). 
Hence from the equivalence of (ii) and (iii) of the graph version of Theorem~(B), we see that the number of edges from $C$ to $V \setminus C$ is odd -- as required for the left end to be odd.

\medskip

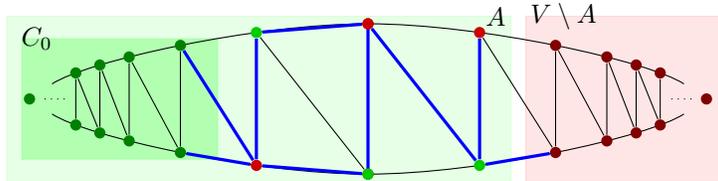
\begin{center}
\begin{tikzpicture}

\draw[draw=green!15, fill=green!10] (-4.8,-1.1) -- (-4.8,1.1) -- (1.9,1.1) -- (1.9,-1.1) -- cycle;  

\draw[draw=red!15, fill=red!10] (2.1,-1.1) -- (2.1,1.1) -- (4.8,1.1) -- (4.8,-1.1) -- cycle; 

\draw[draw=green!35, fill=green!30] (-4.6,-0.8) -- (-4.6,0.8) -- (-2,0.8) -- (-2,-0.8) -- cycle;

\draw (1.7,1.1) node {$A$};
\draw (2.6,1.1) node {$V \setminus A$};

\draw (-4.4,0.8) node {$C_0$};

\draw (-4.5,0) node[circle, inner sep=1.5,fill=green!50!black] {};
\draw (4.5,0) node[circle, inner sep=1.5,fill=red!50!black] {};

\draw[dotted] (-4.3,0) -- (-4,0);
\draw[dotted] (4.3,0) -- (4,0);

\draw[blue, very thick] (0,-1) -- (0,1);

\draw (-4.2,0.2) .. controls (-3.75,0.5) and (-1.5,1) .. (0,1)
node[circle,pos=0.15,fill=green!50!black,inner sep=1.5] (p1) {}
node[circle,pos=0.25,fill=green!50!black,inner sep=1.5] (p2) {}
node[circle,pos=0.35,fill=green!50!black,inner sep=1.5] (p3) {}
node[circle,pos=0.5,fill=green!50!black,inner sep=1.5] (p4) {}
node[circle,pos=0.7,fill=green!80!black,inner sep=1.5] (p5) {} 
;

\draw (-4.2,-0.2) .. controls (-3.75,-0.5) and (-1.5,-1) .. (0,-1)
node[circle,pos=0.15,fill=green!50!black,inner sep=1.5] (n1) {}
node[circle,pos=0.25,fill=green!50!black,inner sep=1.5] (n2) {}
node[circle,pos=0.35,fill=green!50!black,inner sep=1.5] (n3) {}
node[circle,pos=0.5,fill=green!50!black,inner sep=1.5] (n4) {}
node[circle,pos=0.7,fill=red!80!black,inner sep=1.5] (n5) {}
;

\draw[blue, very thick] (n4) -- (n5); 

\draw (0,1) node[circle, inner sep=1.5, fill=red!80!black] (p6) {};
\draw (0,-1) node[circle, inner sep=1.5, fill=red!80!black] (n6) {};

\draw (p1) -- (n1);
\draw (p1) -- (n2);

\draw (p2) -- (n2);
\draw (p2) -- (n3);

\draw (p3) -- (n3);
\draw (p3) -- (n4);

\draw (p4) -- (n4);
\draw[blue,very thick] (p4) -- (n5);

\draw[blue,very thick] (p5) -- (n5);
\draw (p5) -- (n6)
;

\draw[blue,very thick] (p5) -- (p6);
\draw[blue,very thick] (n5) -- (n6);

\draw (4.2,0.2) .. controls (3.75,0.5) and (1.5,1) .. (0,1)
node[circle,pos=0.15,fill=red!50!black,inner sep=1.5] (rp1) {}
node[circle,pos=0.25,fill=red!50!black,inner sep=1.5] (rp2) {}
node[circle,pos=0.35,fill=red!50!black,inner sep=1.5] (rp3) {}
node[circle,pos=0.5,fill=red!50!black,inner sep=1.5] (rp4) {}
node[circle,pos=0.7,fill=red!80!black,inner sep=1.5] (rp5) {}
;

\draw (4.2,-0.2) .. controls (3.75,-0.5) and (1.5,-1) .. (0,-1)
node[circle,pos=0.15,fill=red!50!black,inner sep=1.5] (rn1) {}
node[circle,pos=0.25,fill=red!50!black,inner sep=1.5] (rn2) {}
node[circle,pos=0.35,fill=red!50!black,inner sep=1.5] (rn3) {}
node[circle,pos=0.5,fill=red!50!black,inner sep=1.5] (rn4) {}
node[circle,pos=0.7,fill=green!80!black,inner sep=1.5] (rn5) {}
;

\draw (0,1) node[circle, inner sep=1.5, fill=red!80!black] (rp6) {};
\draw (0,-1) node[circle, inner sep=1.5, fill=green!80!black] (rn6) {};

\draw (rp1) -- (rn1);
\draw (rn1) -- (rp2);

\draw (rp2) -- (rn2);
\draw (rn2) -- (rp3);

\draw (rp3) -- (rn3);
\draw (rn3) -- (rp4);

\draw (rp4) -- (rn4);
\draw (rn4) -- (rp5);

\draw[blue,very thick]  (rp5) -- (rn5);

\draw[blue, very thick]  (rp6) -- (rn5)
;
\draw[blue, very thick]  (rn5) -- (rn4);

\end{tikzpicture}
\end{center}

\bigskip

For the next two examples let $C$ denote the standard `middle thirds' Cantor subset of $I$.

\subsubsection*{Example 2.} The Cantor  bouquet of semi-circles, $\text{CBS}$. The vertices are $C \times \{0\}$ in the plane, along with semi-circular edges centered at the midpoint of each removed open interval.  Note that $\text{CBS}$ is not a Freudenthal compactification of graph.

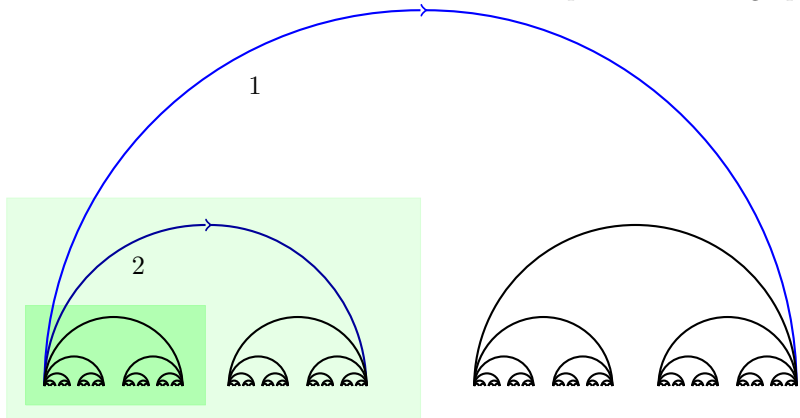
\begin{center}
\begin{tikzpicture}

\draw[fill=green!10,draw=green!15] (-0.5,-0.5) -- (-0.5,2.5) -- (5,2.5) -- (5,-0.5) -- (-0.5,-0.5);

\draw[fill=green!30,draw=green!35] (-0.25,-0.25) -- (-0.25,1.07) -- (2.14,1.07) -- (2.14,-0.25) -- (-0.25,-0.25); 

\draw[thick,blue,-<] (10,0) arc [start angle=0, end angle=90, radius=5];
\draw  (2.8,4) node {$1$};
\draw[thick,blue] (5,5) arc [start angle=90, end angle=180, radius=5];

\draw[thick,blue!60!black,-<] (4.28,0) arc [start angle=0, end angle=90, radius=2.1428];
\draw  (1.25,1.6) node {$2$};
\draw[thick,blue!60!black] (2.1428,2.1428) arc  [start angle=90, end angle=180, radius=2.1428]; 

\semicircle{7.85714285714}{0}{2.14285714286}
\semicircle{0.918367346939}{0}{0.918367346939}
\semicircle{6.63265306122}{0}{0.918367346939}
\semicircle{3.36734693878}{0}{0.918367346939}
\semicircle{9.08163265306}{0}{0.918367346939}
\semicircle{0.393586005831}{0}{0.393586005831}
\semicircle{6.10787172012}{0}{0.393586005831}
\semicircle{2.84256559767}{0}{0.393586005831}
\semicircle{8.55685131195}{0}{0.393586005831}
\semicircle{1.44314868805}{0}{0.393586005831}
\semicircle{7.15743440233}{0}{0.393586005831}
\semicircle{3.89212827988}{0}{0.393586005831}
\semicircle{9.60641399417}{0}{0.393586005831}
\semicircle{0.168679716785}{0}{0.168679716785}
\semicircle{5.88296543107}{0}{0.168679716785}
\semicircle{2.61765930862}{0}{0.168679716785}
\semicircle{8.33194502291}{0}{0.168679716785}
\semicircle{1.218242399}{0}{0.168679716785}
\semicircle{6.93252811329}{0}{0.168679716785}
\semicircle{3.66722199084}{0}{0.168679716785}
\semicircle{9.38150770512}{0}{0.168679716785}
\semicircle{0.618492294877}{0}{0.168679716785}
\semicircle{6.33277800916}{0}{0.168679716785}
\semicircle{3.06747188671}{0}{0.168679716785}
\semicircle{8.781757601}{0}{0.168679716785}
\semicircle{1.66805497709}{0}{0.168679716785}
\semicircle{7.38234069138}{0}{0.168679716785}
\semicircle{4.11703456893}{0}{0.168679716785}
\semicircle{9.83132028322}{0}{0.168679716785}
\semicircle{0.0722913071934}{0}{0.0722913071934}
\semicircle{5.78657702148}{0}{0.0722913071934}
\semicircle{2.52127089903}{0}{0.0722913071934}
\semicircle{8.23555661332}{0}{0.0722913071934}
\semicircle{1.12185398941}{0}{0.0722913071934}
\semicircle{6.83613970369}{0}{0.0722913071934}
\semicircle{3.57083358125}{0}{0.0722913071934}
\semicircle{9.28511929553}{0}{0.0722913071934}
\semicircle{0.522103885286}{0}{0.0722913071934}
\semicircle{6.23638959957}{0}{0.0722913071934}
\semicircle{2.97108347712}{0}{0.0722913071934}
\semicircle{8.68536919141}{0}{0.0722913071934}
\semicircle{1.5716665675}{0}{0.0722913071934}
\semicircle{7.28595228179}{0}{0.0722913071934}
\semicircle{4.02064615934}{0}{0.0722913071934}
\semicircle{9.73493187362}{0}{0.0722913071934}
\semicircle{0.265068126376}{0}{0.0722913071934}
\semicircle{5.97935384066}{0}{0.0722913071934}
\semicircle{2.71404771821}{0}{0.0722913071934}
\semicircle{8.4283334325}{0}{0.0722913071934}
\semicircle{1.31463080859}{0}{0.0722913071934}
\semicircle{7.02891652288}{0}{0.0722913071934}
\semicircle{3.76361040043}{0}{0.0722913071934}
\semicircle{9.47789611471}{0}{0.0722913071934}
\semicircle{0.714880704468}{0}{0.0722913071934}
\semicircle{6.42916641875}{0}{0.0722913071934}
\semicircle{3.16386029631}{0}{0.0722913071934}
\semicircle{8.87814601059}{0}{0.0722913071934}
\semicircle{1.76444338668}{0}{0.0722913071934}
\semicircle{7.47872910097}{0}{0.0722913071934}
\semicircle{4.21342297852}{0}{0.0722913071934}
\semicircle{9.92770869281}{0}{0.0722913071934}
\semicircle{0.0309819887972}{0}{0.0309819887972}
\semicircle{5.74526770308}{0}{0.0309819887972}
\semicircle{2.47996158063}{0}{0.0309819887972}
\semicircle{8.19424729492}{0}{0.0309819887972}
\semicircle{1.08054467101}{0}{0.0309819887972}
\semicircle{6.7948303853}{0}{0.0309819887972}
\semicircle{3.52952426285}{0}{0.0309819887972}
\semicircle{9.24380997714}{0}{0.0309819887972}
\semicircle{0.48079456689}{0}{0.0309819887972}
\semicircle{6.19508028118}{0}{0.0309819887972}
\semicircle{2.92977415873}{0}{0.0309819887972}
\semicircle{8.64405987301}{0}{0.0309819887972}
\semicircle{1.53035724911}{0}{0.0309819887972}
\semicircle{7.24464296339}{0}{0.0309819887972}
\semicircle{3.97933684094}{0}{0.0309819887972}
\semicircle{9.69362255523}{0}{0.0309819887972}
\semicircle{0.22375880798}{0}{0.0309819887972}
\semicircle{5.93804452227}{0}{0.0309819887972}
\semicircle{2.67273839982}{0}{0.0309819887972}
\semicircle{8.3870241141}{0}{0.0309819887972}
\semicircle{1.2733214902}{0}{0.0309819887972}
\semicircle{6.98760720448}{0}{0.0309819887972}
\semicircle{3.72230108203}{0}{0.0309819887972}
\semicircle{9.43658679632}{0}{0.0309819887972}
\semicircle{0.673571386072}{0}{0.0309819887972}
\semicircle{6.38785710036}{0}{0.0309819887972}
\semicircle{3.12255097791}{0}{0.0309819887972}
\semicircle{8.83683669219}{0}{0.0309819887972}
\semicircle{1.72313406829}{0}{0.0309819887972}
\semicircle{7.43741978257}{0}{0.0309819887972}
\semicircle{4.17211366012}{0}{0.0309819887972}
\semicircle{9.88639937441}{0}{0.0309819887972}
\semicircle{0.11360062559}{0}{0.0309819887972}
\semicircle{5.82788633988}{0}{0.0309819887972}
\semicircle{2.56258021743}{0}{0.0309819887972}
\semicircle{8.27686593171}{0}{0.0309819887972}
\semicircle{1.16316330781}{0}{0.0309819887972}
\semicircle{6.87744902209}{0}{0.0309819887972}
\semicircle{3.61214289964}{0}{0.0309819887972}
\semicircle{9.32642861393}{0}{0.0309819887972}
\semicircle{0.563413203682}{0}{0.0309819887972}
\semicircle{6.27769891797}{0}{0.0309819887972}
\semicircle{3.01239279552}{0}{0.0309819887972}
\semicircle{8.7266785098}{0}{0.0309819887972}
\semicircle{1.6129758859}{0}{0.0309819887972}
\semicircle{7.32726160018}{0}{0.0309819887972}
\semicircle{4.06195547773}{0}{0.0309819887972}
\semicircle{9.77624119202}{0}{0.0309819887972}
\semicircle{0.306377444772}{0}{0.0309819887972}
\semicircle{6.02066315906}{0}{0.0309819887972}
\semicircle{2.75535703661}{0}{0.0309819887972}
\semicircle{8.46964275089}{0}{0.0309819887972}
\semicircle{1.35594012699}{0}{0.0309819887972}
\semicircle{7.07022584127}{0}{0.0309819887972}
\semicircle{3.80491971882}{0}{0.0309819887972}
\semicircle{9.51920543311}{0}{0.0309819887972}
\semicircle{0.756190022865}{0}{0.0309819887972}
\semicircle{6.47047573715}{0}{0.0309819887972}
\semicircle{3.2051696147}{0}{0.0309819887972}
\semicircle{8.91945532899}{0}{0.0309819887972}
\semicircle{1.80575270508}{0}{0.0309819887972}
\semicircle{7.52003841937}{0}{0.0309819887972}
\semicircle{4.25473229692}{0}{0.0309819887972}
\semicircle{9.9690180112}{0}{0.0309819887972}

\end{tikzpicture}
\end{center}
The vertex $\mathbf{0}=(0,0)$ is neither odd nor even, and hence $\text{CBS}$ is not Eulerian. Indeed, as indicated on the diagram, there is one (odd) edge connecting all the vertices in the `left half' of the vertices to its complement (the `right half'), but two (even) edges connecting the `left quarter' to its complement.
Similarly we see that  every clopen neighborhood of $\mathbf{0}$ contains two clopen neighborhoods of $\mathbf{0}$ of which  one has an odd  number of edges to its complement, and the other an even number. 

\subsubsection*{Example 3.} The  Cantor bouquet of circles, $\text{CBC}$, can be obtained from the Cantor bouquet of semi-circles by reflecting it in the real axis. One can check that all vertices are even. The diagram illustrates an Eulerian loop in $\text{CBC}$.

\bigskip

\begin{center}
\begin{tikzpicture}

\draw (0.1,2.8) node {$I$};

\draw (0.4,-2) node {$\text{CBC}$};

\draw[very thick,dotted,-|] (0,3.25) -- (10,3.25);

\draw [blue!100!green,very thick,|->] (0,3.25)  --  (1.25,3.25); \draw [blue!100!green, very thick,-|] (1.25,3.25) --  (2.5,3.25); 

\draw [blue!75!green,very thick,->] (2.5,3.25)  --  (2.825,3.25); \draw [blue!75!green, very thick,-|] (2.825,3.25) --  (3.125,3.25);

\draw [blue!62!green,very thick,|->] (3.75,3.25)  --   (4.0625,3.25); \draw [blue!62!green, very thick,-|] (4.0625,3.25) --  (4.375,3.25);

\draw [blue!50!green,very thick,|->] (5,3.25)  --  (6.25,3.25); \draw [blue!50!green, very thick,-|] (6.25,3.25) --  (7.5,3.25); 

\draw [blue!37!green,very thick,->] (7.5,3.25)  --  (7.825,3.25); \draw [blue!37!green, very thick,-|] (7.825,3.25) --  (8.125,3.25);

\draw [blue!6!green,very thick,|->] (8.75,3.25)  --   (9.0625,3.25); \draw [blue!6!green, very thick,-|] (9.0625,3.25) --  (9.375,3.25);

\draw [blue!100!green, very thick, ->] (0,0) .. controls (0,2) and (3,2.5) ..  (5,2.5);
\draw [blue!100!green, very thick] (5,2.5) .. controls (7,2.5) and (10,2) ..  (10,0);

\draw [blue!50!green, very thick, ->] (10,0) .. controls (10,-2) and (7,-2.5) ..  (5,-2.5);
\draw [blue!50!green, very thick] (5,-2.5) .. controls (3,-2.5) and (0,-2) ..  (0,0);

\draw [blue!75!green, very thick, ->] (10,0) .. controls (10,-1) and (8.7429,-1) ..  (7.85714,-1);
\draw [blue!75!green, very thick] (7.85714,-1) .. controls (7,-1) and (5.71428,-1) ..  (5.714285714,0);

\draw [blue!62!green, very thick, ->] (5.714285714,0)  .. controls (5.71428,1) and (7,1)  ..  (7.85714,1);
\draw [blue!62!green, very thick] (7.85714,1)  .. controls  (8.7429,1) and (10,1) .. (10,0);

\draw [blue!37!green, very thick, ->] (0,0)  .. controls (0,1) and (1.2857,1)  ..  (2.142857,1);
\draw [blue!37!green, very thick] (2.142857,1)  .. controls  (3,1) and (4.285,1) .. (4.2857,0);

\draw [blue!6!green, very thick, ->] (4.2857,0)   .. controls (4.285,-1) and (3,-1)  .. (2.142857,-1);
\draw [blue!6!green, very thick] (2.142857,-1)   .. controls (1.2857,-1) and (0,-1)  ..  (0,0);

\ellipse{0.918367346939}{0}{0.918367346939}
\ellipse{6.63265306122}{0}{0.918367346939}
\ellipse{3.36734693878}{0}{0.918367346939}
\ellipse{9.08163265306}{0}{0.918367346939}
\ellipse{0.393586005831}{0}{0.393586005831}
\ellipse{6.10787172012}{0}{0.393586005831}
\ellipse{2.84256559767}{0}{0.393586005831}
\ellipse{8.55685131195}{0}{0.393586005831}
\ellipse{1.44314868805}{0}{0.393586005831}
\ellipse{7.15743440233}{0}{0.393586005831}
\ellipse{3.89212827988}{0}{0.393586005831}
\ellipse{9.60641399417}{0}{0.393586005831}
\ellipse{0.168679716785}{0}{0.168679716785}
\ellipse{5.88296543107}{0}{0.168679716785}
\ellipse{2.61765930862}{0}{0.168679716785}
\ellipse{8.33194502291}{0}{0.168679716785}
\ellipse{1.218242399}{0}{0.168679716785}
\ellipse{6.93252811329}{0}{0.168679716785}
\ellipse{3.66722199084}{0}{0.168679716785}
\ellipse{9.38150770512}{0}{0.168679716785}
\ellipse{0.618492294877}{0}{0.168679716785}
\ellipse{6.33277800916}{0}{0.168679716785}
\ellipse{3.06747188671}{0}{0.168679716785}
\ellipse{8.781757601}{0}{0.168679716785}
\ellipse{1.66805497709}{0}{0.168679716785}
\ellipse{7.38234069138}{0}{0.168679716785}
\ellipse{4.11703456893}{0}{0.168679716785}
\ellipse{9.83132028322}{0}{0.168679716785}
\ellipse{0.0722913071934}{0}{0.0722913071934}
\ellipse{5.78657702148}{0}{0.0722913071934}
\ellipse{2.52127089903}{0}{0.0722913071934}
\ellipse{8.23555661332}{0}{0.0722913071934}
\ellipse{1.12185398941}{0}{0.0722913071934}
\ellipse{6.83613970369}{0}{0.0722913071934}
\ellipse{3.57083358125}{0}{0.0722913071934}
\ellipse{9.28511929553}{0}{0.0722913071934}
\ellipse{0.522103885286}{0}{0.0722913071934}
\ellipse{6.23638959957}{0}{0.0722913071934}
\ellipse{2.97108347712}{0}{0.0722913071934}
\ellipse{8.68536919141}{0}{0.0722913071934}
\ellipse{1.5716665675}{0}{0.0722913071934}
\ellipse{7.28595228179}{0}{0.0722913071934}
\ellipse{4.02064615934}{0}{0.0722913071934}
\ellipse{9.73493187362}{0}{0.0722913071934}
\ellipse{0.265068126376}{0}{0.0722913071934}
\ellipse{5.97935384066}{0}{0.0722913071934}
\ellipse{2.71404771821}{0}{0.0722913071934}
\ellipse{8.4283334325}{0}{0.0722913071934}
\ellipse{1.31463080859}{0}{0.0722913071934}
\ellipse{7.02891652288}{0}{0.0722913071934}
\ellipse{3.76361040043}{0}{0.0722913071934}
\ellipse{9.47789611471}{0}{0.0722913071934}
\ellipse{0.714880704468}{0}{0.0722913071934}
\ellipse{6.42916641875}{0}{0.0722913071934}
\ellipse{3.16386029631}{0}{0.0722913071934}
\ellipse{8.87814601059}{0}{0.0722913071934}
\ellipse{1.76444338668}{0}{0.0722913071934}
\ellipse{7.47872910097}{0}{0.0722913071934}
\ellipse{4.21342297852}{0}{0.0722913071934}
\ellipse{9.92770869281}{0}{0.0722913071934}
\ellipse{0.0309819887972}{0}{0.0309819887972}
\ellipse{5.74526770308}{0}{0.0309819887972}
\ellipse{2.47996158063}{0}{0.0309819887972}
\ellipse{8.19424729492}{0}{0.0309819887972}
\ellipse{1.08054467101}{0}{0.0309819887972}
\ellipse{6.7948303853}{0}{0.0309819887972}
\ellipse{3.52952426285}{0}{0.0309819887972}
\ellipse{9.24380997714}{0}{0.0309819887972}
\ellipse{0.48079456689}{0}{0.0309819887972}
\ellipse{6.19508028118}{0}{0.0309819887972}
\ellipse{2.92977415873}{0}{0.0309819887972}
\ellipse{8.64405987301}{0}{0.0309819887972}
\ellipse{1.53035724911}{0}{0.0309819887972}
\ellipse{7.24464296339}{0}{0.0309819887972}
\ellipse{3.97933684094}{0}{0.0309819887972}
\ellipse{9.69362255523}{0}{0.0309819887972}
\ellipse{0.22375880798}{0}{0.0309819887972}
\ellipse{5.93804452227}{0}{0.0309819887972}
\ellipse{2.67273839982}{0}{0.0309819887972}
\ellipse{8.3870241141}{0}{0.0309819887972}
\ellipse{1.2733214902}{0}{0.0309819887972}
\ellipse{6.98760720448}{0}{0.0309819887972}
\ellipse{3.72230108203}{0}{0.0309819887972}
\ellipse{9.43658679632}{0}{0.0309819887972}
\ellipse{0.673571386072}{0}{0.0309819887972}
\ellipse{6.38785710036}{0}{0.0309819887972}
\ellipse{3.12255097791}{0}{0.0309819887972}
\ellipse{8.83683669219}{0}{0.0309819887972}
\ellipse{1.72313406829}{0}{0.0309819887972}
\ellipse{7.43741978257}{0}{0.0309819887972}
\ellipse{4.17211366012}{0}{0.0309819887972}
\ellipse{9.88639937441}{0}{0.0309819887972}
\ellipse{0.11360062559}{0}{0.0309819887972}
\ellipse{5.82788633988}{0}{0.0309819887972}
\ellipse{2.56258021743}{0}{0.0309819887972}
\ellipse{8.27686593171}{0}{0.0309819887972}
\ellipse{1.16316330781}{0}{0.0309819887972}
\ellipse{6.87744902209}{0}{0.0309819887972}
\ellipse{3.61214289964}{0}{0.0309819887972}
\ellipse{9.32642861393}{0}{0.0309819887972}
\ellipse{0.563413203682}{0}{0.0309819887972}
\ellipse{6.27769891797}{0}{0.0309819887972}
\ellipse{3.01239279552}{0}{0.0309819887972}
\ellipse{8.7266785098}{0}{0.0309819887972}
\ellipse{1.6129758859}{0}{0.0309819887972}
\ellipse{7.32726160018}{0}{0.0309819887972}
\ellipse{4.06195547773}{0}{0.0309819887972}
\ellipse{9.77624119202}{0}{0.0309819887972}
\ellipse{0.306377444772}{0}{0.0309819887972}
\ellipse{6.02066315906}{0}{0.0309819887972}
\ellipse{2.75535703661}{0}{0.0309819887972}
\ellipse{8.46964275089}{0}{0.0309819887972}
\ellipse{1.35594012699}{0}{0.0309819887972}
\ellipse{7.07022584127}{0}{0.0309819887972}
\ellipse{3.80491971882}{0}{0.0309819887972}
\ellipse{9.51920543311}{0}{0.0309819887972}
\ellipse{0.756190022865}{0}{0.0309819887972}
\ellipse{6.47047573715}{0}{0.0309819887972}
\ellipse{3.2051696147}{0}{0.0309819887972}
\ellipse{8.91945532899}{0}{0.0309819887972}
\ellipse{1.80575270508}{0}{0.0309819887972}
\ellipse{7.52003841937}{0}{0.0309819887972}
\ellipse{4.25473229692}{0}{0.0309819887972}
\ellipse{9.9690180112}{0}{0.0309819887972}

\end{tikzpicture}
\end{center}

\medskip

Suppose $f\colon I \to X$ is a standard path in a graph-like continuum $X$ with vertices $V$ and (open) edges $(e_n)_n$. Then $f^{-1}(V)$ is a closed nowhere dense subset of $I$, and its complement, $f^{-1}\left( \bigcup_n e_n\right)$ is dense and a disjoint union of open intervals. This countable family, $\{ f^{-1} (e_n) : n \in \N\}$ inherits an order from the order on $I$.
So to every path $f$ we can associate a countable linear order $L_f$, which we informally call the \emph{shape} of $f$. 

To illustrate this, consider $L=L_f$ where $f$ is the Eulerian loop in the Cantor bouquet of circles diagrammed above.
Then $f$ traverses the top edge from left to right, covers the right-hand copy of $\text{CBC}$, traverses the bottom edge from right to left, and then covers the left-hand copy of $\text{CBC}$. 
So $L$ satisfies the equation $L=1+L+1+L$. It follows that $L$ is an infinite ordinal. Thus $L=1+L$, and we see $L=L+L$.
The first infinite ordinal which is a fixed point under addition of linear orders is the ordinal $\omega^\omega$. Hence $L=\omega^\omega$. 

We now see how to construct for each countable linear order $L$ a graph-like continuum $X_L$ with an Eulerian loop $f$ so that $L_f=L$.
To do so recall: every countable linearly ordered set $L$ can be realized (is order isomorphic to) a countable family of disjoint open subintervals of $I$, with dense union. For further material on the interaction of linear orders and graph-like compacta, see \cite[\S4]{infinitematroids}.

Given a line segment, $S$, in the plane the `circle with diameter $S$' is the circle with center the midpoint of the line segment, and radius half the length of the segment.

\subsubsection*{Example 4.}\label{XL} Let $L$ be a countable linear order. Fix a family $\mathcal{U}$ of pairwise disjoint open subintervals of $I$, with dense union, which is order isomorphic to $L$.
Define $X_L$ to be the subspace of the plane obtained by starting with $X=I \times \{0\}$, and for each $U$ in $\mathcal{U}$, remove $U \times \{0\}$ from $X$ and add the circle with diameter $U \times \{0\}$.

The Eulerian loops on $X_L$ are naturally bijective with all functions $\rho : L \to \{\pm 1\}$.
To see this take any $\rho : L \to \{\pm 1\}$. Since $\mathcal{U}$ and $L$ are isomorphic we can think that the domain of $\rho$ is actually $\mathcal{U}$. 
Define $g_\rho : [0,1] \to X_L$ by requiring (i) $g(t)=t$  on $I \setminus \bigcup \mathcal{U}$, and (ii) on $U$ in $\mathcal{U}$ the path $g$ traverses the top (resp. bottom) semi-circle in $X_L$ corresponding to $U$ if $\rho (U)=+1$ (resp. $\rho(U)=-1$).
Now define $f_\rho$ -- the desired Eulerian loop -- by $f_\rho(t)=g_\rho(2t)$ on $[0,1/2]$ and $f_\rho(t)=g_{-\rho}(2-2t)$ on $[1/2,1]$. Informally, on $[0,1/2]$ the path $f_\rho$ travels from left to right along $X_L$ crossing the circles by either taking the upper or lower semi-circles depending on $\rho$; and then on $[1/2,1]$ it travels across $X_L$ from right to left taking the opposite upper/lower semi-circles than before. 
Every Eulerian loop arises in this way, and observe that they all have the same shape, $L$. 

The following diagram depicts $X_Q$ where $Q$ is the linearly ordered set of dyadic rationals in $(0,1)$. Recall that $Q$ is order isomorphic to the rationals, $\Q$.

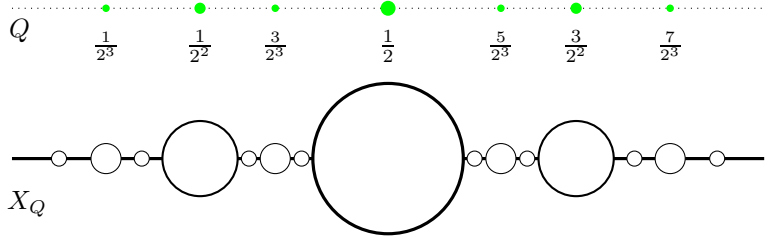
\bigskip

\begin{center}
\begin{tikzpicture}

\draw (0.1,1.7) node {$Q$};

\draw (0.2,-0.6) node {$X_Q$};

\draw[thin, dotted] (0,2) -- (10,2);

\draw[very thick] (0,0) -- (10,0);

\draw[fill=white, very thick] (5,0) circle [radius=1.0]; 
\draw (5,1.5) node  {$\frac{1}{2}$}; 
\draw (5,2) node[circle,fill=green,inner sep=2] {};

\draw[fill=white, thick] (2.5,0) circle [radius=0.5];
\draw (2.5,1.5) node  {$\frac{1}{2^2}$}; 
\draw (2.5,2) node[circle,fill=green,inner sep=1.5] {};
\draw[fill=white, thick] (7.5,0) circle [radius=0.5];
\draw (7.5,1.5) node  {$\frac{3}{2^2}$}; 
\draw (7.5,2) node[circle,fill=green,inner sep=1.5] {};

\draw[fill=white] (1.25,0) circle [radius=0.2];
\draw (1.25,1.5) node  {\small $\frac{1}{2^3}$}; 
\draw (1.25,2) node[circle,fill=green,inner sep=1] {};
\draw[fill=white] (3.5,0) circle [radius=0.2];
\draw (3.5,1.5) node  {\small $\frac{3}{2^3}$}; 
\draw (3.5,2) node[circle,fill=green,inner sep=1] {};
\draw[fill=white] (6.5,0) circle [radius=0.2];
\draw (6.5,1.5) node  {\small $\frac{5}{2^3}$}; 
\draw (6.5,2) node[circle,fill=green,inner sep=1] {};
\draw[fill=white] (8.75,0) circle [radius=0.2];
\draw (8.75,1.5) node  {\small $\frac{7}{2^3}$}; 
\draw (8.75,2) node[circle,fill=green,inner sep=1] {};

\draw[fill=white,thin] (0.625,0) circle [radius=0.1];
\draw[fill=white,thin] (1.725,0) circle [radius=0.1];
\draw[fill=white,thin] (3.15,0) circle [radius=0.1];
\draw[fill=white,thin] (3.85,0) circle [radius=0.1];
\draw[fill=white,thin] (6.15,0) circle [radius=0.1];
\draw[fill=white,thin] (6.85,0) circle [radius=0.1];
\draw[fill=white,thin] (8.275,0) circle [radius=0.1];
\draw[fill=white,thin] (9.375,0) circle [radius=0.1];

\end{tikzpicture}
\end{center}
The graph-like continuum $X_Q$ provides an example of the difficulties involved in na\"ively trying to lift arguments for graphs to graph-like continua. 
In the standard proof of  Theorem~(B) for graphs one moves from (iv) `the edges of the graph can be decomposed into disjoint cycles' to (i) `there is an Eulerian circuit' by amalgamating the cycles, one after another to form the circuit. Notice that in $X_Q$ there is a canonical decomposition of $X_Q$ into edge disjoint circles -- namely the circles in the definition of $X_Q$. But these circles are \emph{pairwise disjoint}. Hence there is no natural method of amalgamating them into an Eulerian loop for $X_Q$.

\subsubsection*{Example 5.} The Hawaiian earring, $H$, is also Eulerian. Unlike the $X_L$ examples above,  \emph{every} countable linear order can be realized as the $L_f$ of an Eulerian loop.
\begin{wrapfigure}{r}{4cm}

\begin{tikzpicture}
\draw[thick] (0,0) arc [start angle=-90, end angle=270, radius=2];
\draw[thick] (0,0) arc [start angle=-90, end angle=270, radius=1.5];
\draw[thick] (0,0) arc [start angle=-90, end angle=270, radius=1.1];
\draw[thick] (0,0) arc [start angle=-90, end angle=270, radius=0.75];
\draw[thick] (0,0) arc [start angle=-90, end angle=270, radius=0.5];
\draw[thick] (0,0) arc [start angle=-90, end angle=270, radius=0.35];

\draw[dotted,red,very thick] (0,0) -- (0,0.5);
\end{tikzpicture}

\end{wrapfigure} 
  Write $H$ as $\mathbf{0}=(0,0)$ (the sole  vertex) and the union of circles in the plane $C_n$, for $n \in \N$, where $C_n$ has radius $1/n$ and is tangential at $\mathbf{0}$ to the $x$-axis.  
  
We  can identify the Eulerian loops in the Hawaiian earring as follows. For any countable linear order $L$ and function $\rho \colon L \to \N \times \{\pm 1\}$ such that $\pi_1 \circ \rho \colon L \to \N$ is a bijection, there is a naturally corresponding Eulerian loop $f_\rho$ of $H$.
Indeed, given $L$ and $\rho$, let $\mathcal{U}$ be a family of pairwise disjoint open subintervals of $I$, with dense union, which is order isomorphic to $L$  (and identify them). 
Define $f_\rho$ to have value $\mathbf{0}$ on the complement of $\bigcup \mathcal{U}$, and on $U$ in $\mathcal{U}$, writing $\rho(U)=(n,i)$, it should traverse $C_n$ clockwise (respectively, anticlockwise) if $i=+1$ (respectively, $i=-1$). 
 One can check all Eulerian loops arise this way.

\section{Properties and Characterizations of Graph-like continua}\label{sec:repchar}

\subsection{Basic Properties}
Most of the following basic properties of graph-like spaces are well-known, see e.g.\ \cite{thomassenvella}. Nonetheless, it might be helpful to give a self-contained outline of the most important properties we use. 

Let $(X,V,E)$ be a compact graph-like space. We often identify the label, $e$, of an edge, with the subspace $e\times (0,1)$ of $X$.
Note that since $V$ is zero-dimensional, for every edge $e$, the closure, $\closure{e}$, of $e$ adds at most two vertices -- the ends of the edge -- and $\closure{e}$ is homeomorphic to the circle, $S^1$, or $I=[0,1]$. With this in mind, our definition of compact graph-like space is the same as the original in \cite{thomassenvella}.

A \emph{separation} $(A,B)$ of a graph like space $X$ is a partition of $V(X)$ into two disjoint clopen subsets. The \emph{cut induced by the separation} $(A,B)$ is  set of edges with one end vertex in $A$ and the other in $B$, denoted by $E(A,B)$. More generally, we call a subset $F \subset E$ a cut if there is a separation $(A,B)$ of $X$ such that $F = E(A,B)$. A \emph{multi-cut} is a partition $\mathcal{U}=\Set{U_1, U_2, \ldots , U_n}$ of $V(X)$ into pairwise disjoint clopen sets. For each two $U_i, U_j$, not necessarily different, $E(U_i,U_j)$ denotes the set of edges with one endpoint in $U_i$ and the other endpoint in $U_j$. By $X[U_i]$ we denote the \emph{induced subspace} of $X$, i.e.\ the closed graph-like subspace with vertex set $U_i$ and edge set $E(U_i,U_i)$. Finally, a clopen subset $U \subset V(X)$ is called a \emph{region} if the induced subspace $X[U]$ is connected.

\begin{lemma}
\label{finitecuts}
In a compact graph-like space, all cuts are finite.
\end{lemma}
\begin{proof}
Suppose there is an infinite cut $F=\set{f_n}:{n \in \N}$ induced by a separation $(A,B)$ of a graph-like space $X$. Then $A$ and $B$ are disjoint closed subsets of $X$, so by normality there are disjoint open subsets $U \supseteq A$ and $V \supseteq B$. Since edges are connected, there exist $x_n \in f_n \setminus \p{U \cup V}$ for all $n$. It follows that $\set{x_n}:{n \in \N}$ is an infinite closed discrete subset, contradicting compactness.
\end{proof}

\begin{lemma}
\label{basisregion}
Let $X$ be a compact graph-like space. For every vertex $v$ of $X$ and any open neighborhood $U$ of $v$, there is a clopen $C \subset V(X)$ such that $v \in C$ and $X[C] \subset U$. Moreover, if $X$ is connected, then $C$ can be chosen to be a region.
\end{lemma}
\begin{proof}
Since $V(X)$ is totally disconnected we have 
$$\singleton{v} = \bigcap \set{X[A]}:{(A,B) \text{ a separation of } X, \; v \in A}.$$
Now $ \bigcap X[A] \subset U$ and compactness implies that there is a finite subcollection $A_1, \ldots, A_n$ such that for the clopen set $B=A_1\cap \cdots \cap A_n$ we have
$$v \in X[B] = X[A_1] \cap \cdots \cap X[A_n]\subset U.$$
For the moreover part, since $E(B, V \setminus B)$ is finite by Lemma~\ref{finitecuts}, it follows from connectedness of $X$ that $X[B]$ consists of finitely many connected components, say $X[B]=X[C_1] \oplus \cdots \oplus X[C_k]$, one of which contains the vertex $v$. This is our desired region $C$.
\end{proof}

\begin{defn}
Let $X$ be a graph-like space and $\mathcal{U}$ be a multi-cut of $X$. The \emph{multi-graph associated with $\mathcal{U}$} is the quotient 
space $G(\script{U})=X / \set{X[U]}:{U \in \script{U}}$. The map $\pi_\script{U} \colon X \to G(\script{U})$ denotes the corresponding quotient map.
\end{defn}

We remark that $G(\script{U})$ is indeed a finite multi-graph. The identified $X[U]$ form a finite collection of vertices, which are connected by finitely many edges (see Lemma~\ref{finitecuts}). The degree of $\pi_\mathcal{U}(U_i)$ in $G(\mathcal{U})$ is given by $\cardinality{E(U_i, V \setminus U_i)} < \infty.$ Our next proposition gathers properties of graphs associated with multi-cuts.

\begin{prop}
\label{prop:invlimhyp}
Let $X$ be a graph-like compact space. Then
\begin{enumerate}
\item $X$ is connected if and only if $G(\script{U})$ is connected for all multi-cuts $\script{U}$ of $X$.
\item All cuts of $X$ are even if and only if all vertices in $G(\script{U})$ have even degrees for all multi-cuts $\script{U}$ of $X$.
\end{enumerate}
\end{prop}

\begin{proof}
(1) If $X$ is connected, then connectedness of $G\p{\mathcal{U}}$ follows from the fact that it is the continuous image of $X$. Conversely, a disconnection of $X$ gives rise to a $G\p{\mathcal{U}}$ which is the empty graph on two vertices.

(2) If every cut of $X$ is even, then the above degree considerations show that every vertex in $G\p{\mathcal{U}}$ has even degree. And conversely, any odd cut of $X$ gives rise to a graph $G\p{\mathcal{U}}$ on two vertices of odd degree.
\end{proof}

Recall that a \emph{standard subspace} $Y$ of a graph-like space $X$ is a closed subspace that contains all edges it intersects (i.e.\ whenever $e \cap Y \neq \emptyset$ then $e \subset Y$). Standard subspaces of graph-like spaces are again graph-like. Write $E(Y)$ for the collection of edges of $Y$.

\begin{lemma}
\label{standard}
Let $X$ be a graph-like space and $C \subseteq X$ a copy of a topological circle. Then $C$ is a standard subspace.
\end{lemma}

\begin{proof}
Assume, by contradiction, that there exists $e\in E(X)$ such that $e\cap C\neq\emptyset$ and that $e\not\subset C$. Let $y\in e\setminus C$. 
Then there exist $x_0\in C$ with the properties that the arc $[x_0,y]$ is a subset of $e$ and $[x_0,y]\cap C=\{x_0\}$. Observe that $x_0\not\in V$. Let $U$
be an open set containing $x_0$ such that $U\cap V=\emptyset$. Let $\alpha$ be the component in $[x_0,y]$ of $x_0$ contained in $U$ and $\beta$ be the component
in $C$ of $x_0$ contained in $U$. Then $\alpha\cup\beta$ contains a triod and $\alpha\cup\beta\subset X\setminus V$ which is a contradiction to the fact 
that $X\setminus V\cong E\times (0,1)$ contains no triods.
\end{proof}

\begin{lemma}
\label{lemma:evencuts}
Let $X$ be a graph-like space and $C \subseteq  X$ a copy of a topological circle. Then $E(C) \cap F$ is finite and even for all cuts $F=E(A,B)$ of $X$.
\end{lemma}

\begin{proof}
By Lemma~\ref{standard} we may assume $X=C$. Let $F=E(A,B)$. That $F$ is finite is immediate from Lemma~\ref{finitecuts}, so we only need to prove that $\cardinality{F}$ is even.

Let $C[A]$ (resp. $C[B]$) be the standard subspace containing $A$ (resp. $B$) and all edges with both endpoints in $A$ (resp. $B$). Observe that
\begin{enumerate}[label=(\alph*)]
\item $C=C[A]\cup F\cup C[B]$, and
\item $C[A]$ and $C[B]$ have finitely many components.
\end{enumerate}
Let $A_1, \dots, A_r$ be the components of $C[A]$ and $B_1, \dots, B_s$ be the components of $C[B]$. These components induce a multi-cut, $\mathcal{U}=\Set{U_{A_1}, \dots, U_{A_r}, U_{B_1}, \dots, U_{B_s}}$, of the vertices of $C$ where $U_{A_i}$ (resp. $U_{B_i}$) consists of all vertices contained in $A_i$ (resp. $B_i$). Then $G\p{\mathcal{U}}$, the multi-graph associated with $\mathcal{U}$, is a cycle whose edges are the elements of $F$ and whose vertices are the equivalence classes containing the sets $U_{A_1},\dots, U_{B_s}$. Observe that the sets $\mathcal{A}=\Set{U_{A_1},\dots, U_{A_r}}$ and $\mathcal{B}=\Set{U_{B_1}, \dots, U_{B_s}}$ give a 2-coloring of the vertices of $G\p{\mathcal{U}}$. Hence $G\p{\mathcal{U}}$ has an even number of edges, i.e.\ $\cardinality{F}$ is even.
\end{proof}

\subsection{Characterizations and Representations}
In this section we prove Theorem~(A). The equivalence of (i) and (iii) is given by Proposition~\ref{pro:grlkcomreg}, the equivalence of (i) and (ii) is Theorem~\ref{chargraphlike}, while the equivalence of (i), (iv) and (iv)' follows from Theorems~\ref{inverselimit} and~\ref{gralikeasinverse}. Compact graph-like spaces were explicitly defined to include standard subspaces of the Freudenthal compactification of locally finite graphs. Theorem~\ref{standardsubspacetheorem} provides the converse, establishing equivalence of (i) and (v).

Recall that a continuum $X$ is \emph{regular} if it has a basis of open sets, each with finite boundary, and it is called \emph{completely regular} if each non-degenerate subcontinuum of $X$ has non-empty interior in $X$, see ~\cite[Page 1176]{Koenigsberg}. A continuum is \emph{hereditarily locally connected (hlc)} if every subcontinuum is locally connected, and \emph{finitely Souslian} if each sequence of pairwise disjoint subcontinua forms a null-sequence, i.e. the diameters of the subcontinua converge to zero. It is known that for continua
\begin{enumerate}
\item[($\ddagger$)] completely regular $\Rightarrow$ regular $\Rightarrow$ finitely  Souslian $\Rightarrow$ hlc $\Rightarrow$ arc-connected.
\end{enumerate}
For the first three implications, see \cite[Proposition 1.1]{kra}.

\begin{lemma}
Every compact graph-like space is regular.
\label{lemma:gl-regular}
\end{lemma}

\begin{proof}
Let $X$ be a compact graph-like space, $p\in X$, and $U$ be an open of $X$ set such that $p\in U$. We will show that there is an open set $O$ with finite boundary such that $p\in O\subseteq U$. 

The case when $p$ is in the interior of an edge follows from the fact that the set of edges is discrete. So we may assume $p\in V$. For this case let $X[B]$ as in the proof of Lemma~\ref{basisregion}, then $p\in X[B]\subseteq U$. Now for each $e\in E(B, V\setminus B)$, let $(v,x_e)$ be a subarc of $e$ such that $(v,x_e)\subseteq U$ and such that $v\in B$. Since cuts are finite, then there are only finitely many of these arcs. The desire open set $O$ is then $X[B]\cup\{(v,x_e) : e\in E(B,V\setminus B)\}$ as its boundary is the set $\{x_e: e\in E(B,V\setminus B)\}$.
\end{proof}

\begin{cor}
\label{souslian}
Every graph-like continuum is finitely Souslian, hereditarily locally connected and arc-connected.
\label{cor:glislc}
\end{cor}
\begin{proof}
By Lemma~\ref{lemma:gl-regular} and ($\ddagger$), this is a consequence of regular.

For a direct proof that graph-like continua are finitely Souslian, suppose for a contradiction that $\set{A_i}:{i \in \N}$ forms a sequence of disjoint subcontinua of $X$ with non-vanishing diameter. It follows from the sequential compactness of the hyperspace of subcontinua, \cite[Corollary 4.18]{Nadler}, that there is a subsequence $A_{i_j}$ such that
$A=\lim_{j \to \infty} A_{i_j} = \closure{\bigcup_{j} A_{i_j}} \setminus \bigcup_{j} A_{i_j}$ 
is a non-trivial subcontinuum of $X$.
But since edges are open, we also have that $A \subset V(X)$, so is totally disconnected, a contradiction.

For a direct proof that graph-like continua are hlc, see Lemma~\ref{basisregion}.
\end{proof}

In particular, noting that a compact graph-like space has at most countably many edges (as they form a collection of disjoint open subsets), it follows that the edges of $X$ form a null-sequence, i.e.\ $\lim_{n\to \infty} \diam{e_n}=0$. Here, for a subset $A$ of a metric space, we denote by $\diam{A}$ the diameter of $A$. 

In the next theorem we use the following notation. For a subspace $A \subset X$ we denote by $\bd{A}$ its boundary. A subarc $A \subset X$ is called \emph{free} if $A$ removed its endpoints is open in $X$.

\begin{theorem}[{\cite[Theorem 1.3]{kra}}]
A continuum $X$ is completely regular if and only if there exists a $0$-dimensional compact subset $F$ of $X$ and a finite or countable null sequence of free arcs $A_1, A_2, \dots$ in $X$ such that $$X=F\cup\left(\bigcup\{A_n : n\geq 1\}\right) \text{ and } A_j\cap F = \bd{A_j}$$
for each $j\geq 1$
\label{thm:compreg}
\end{theorem}

Observe that Theorem~\ref{thm:compreg} implies that every completely regular continuum is a graph-like space. Conversely, if $X$ is a graph-like continuum, then the set of vertices $V$ is zero-dimensional. Also by Corollary~\ref{cor:glislc}, $E(X)$ forms a null sequence. By Theorem~\ref{thm:compreg}, $X$ is a completely regular continuum. 
  
\begin{prop}
Let $X$ be a continuum. Then $X$ is completely regular if and only if $X$ is a graph-like space.
\label{pro:grlkcomreg}
\end{prop}

Recall that a graph can be characterized in terms of order: a continuum is a graph if and only if every point has finite order, and all but finitely many points have order $2$, \cite[Theorem 9.10 \& 9.13]{Nadler}.

\begin{theorem}[Graph-like Characterization]
\label{chargraphlike}
A continuum is graph-like if and only if it is regular and has a closed zero-dimensional subset $V$ such that all points outside of $V$ have order $2$.
\end{theorem}

\begin{proof} Sufficiency follows from the definition of graph-like and Lemma \ref{lemma:gl-regular}.

For the necessity, first observe that regular implies local connectedness. Let $V \subset X$ be a closed zero-dimensional collection of points in $X$ such that all points outside of $V$ have order $2$. By local connectedness, all components of $X \setminus V$ are open subsets of $X$. In particular, we have at most countably many components, and each component is non-trivial, non-compact, and consists exclusively of points of order $2$. So each component is homeomorphic to an open interval.
So all that remains to show for graph-like is that the closure of each edge is compact, which is automatic.
\end{proof}
 
\begin{cor}[Canonical Representation of Graph-like Spaces]
\label{repstd}
Let $X \not\cong S^1$ be a graph-like continuum. Then there is a unique minimal set $V \subset X$ which witnesses that $X$ is a graph-like space. We call $(X,V,E)$ the standard representation of $X$.
\end{cor}

\begin{proof}
Let $\set{V_s}:{s \in S}$ be the collection of all subsets of $X$ which witness that $X$ is graph-like. We claim that $V = \bigcap_{s \in S} V_s$ is also a vertex set. 

Clearly, $V$ is closed and zero-dimensional. Further, if $x \notin V$, then $x \notin V_s$ for some $s \in S$, so $x$ has order 2. So either $V$ is empty, in which case $X \cong S^1$; or $V$ is non-empty, in which case every component of $X \setminus V$ is non-compact, open, and consists of points of order $2$, so is homeomorphic to an open interval.
\end{proof}

%
%

Our next theorem has been proved, for completely regular continua, by Nikiel, \cite[3.8]{inverselimitsnikiel}. We reprove his theorem here (and extend it to graph-like compacta), phrased for convenience in the language of graph-like continua. 

\begin{theorem}[Inverse Limit Representation]
\label{inverselimit}
Every graph-like compact space $X$ can be represented as an inverse limit of multi-graphs $G_n$ ($n \in \mathbb{N}$) with onto simplicial bonding maps that have non-trivial fibres at vertices only, such that
\begin{enumerate}
\item $X$ connected if and only if $G_n$ is connected for all $n$,\label{connectedness} and
	\item all cuts $E\p{A,B}$ in $X$ are even $\Leftrightarrow$ all vertices in $G_n$ are even  for all $n$. \label{item:evencuts}
\end{enumerate}
Moreover, if $X$ is connected, then the bonding maps can additionally be chosen monotone.
\end{theorem}
\begin{proof}
Let $X$ be a graph-like continuum with vertex set $V$ and edge set $E$. Without loss of generality,  $X$ contains no loops, as otherwise we can subdivide each edge once (this does not change the homeomorphism type of $X$, and the new edge set is still a dense open subset, so the new vertex set is a compact, zero-dimensional subspace as required).

Since $V$ is a compact, zero-dimensional metrizable space, we can find, as in Lemma~\ref{basisregion}, a sequence of multi-cuts $\set{\mathcal{U}_n}:{n \in \N}$ such that 
\begin{enumerate}[label=(\alph*)]
\item $\mathcal{U}_{n+1}$ is a refinement of $\mathcal{U}_n$,
\item $\bigcup_{n \in \N} \mathcal{U}_n$ forms a basis for $V(X)$, and
\end{enumerate}
Writing $\mathcal{U}_n=\{U^n_1, U^n_2, \ldots , U^n_{i(n)}\}$ we observe that every $v \in V$ has a unique description in terms of $\Set{v}=\bigcap_{n \in \N} U^n_{l(v)}$ and that conversely, for every nested sequence of cut elements, there is precisely one vertex in $\bigcap_{n \in N} U^n_{l_n}$ by compactness and $(b)$.

\smallskip
\noindent
{\bf The inverse system:} Let $\set{\mathcal{U}_n}:{n \in \N}$ be as above. To simplify notation, let $q_n$ stand for $\pi_{\mathcal{U}_n}$. For each $n\in\mathbb{N}$ let $f_n \colon G\p{\mathcal{U}_{n+1}}\to G\p{\mathcal{U}_n}$ be defined as
$$f_n(x)=q_n\p{q^{-1}_{n+1}(x)} \text{ for all }x\in G\p{\mathcal{U}_{n+1}}.$$ Observe that if $U_i^{n+1},U_j^{n+1}\subset U_s^n$, 
\begin{enumerate}[label=(\roman*)]
\item then $f\p{q_{n+1}(U_i^{n+1})}=f\p{q_{n+1}(U_j^{n+1})}=q_n(U_s^n)$;
\item  and if $e\in E(U_i^{n+1},U_j^{n+1})$; in particular $e\in E(U_s^n,U_s^n)$, then $f_n(e)=q_n(U_s^n)$.
\end{enumerate}
In particular, each $f_n$ is an onto simplicial map with non-trivial fibres only at vertices of $G\p{\script{U}_n}$. Then $\{G\p{\mathcal{U}_n},f_n\}_{n\in\N}$ is an inverse sequence of multi-graphs. Hence, its inverse limit is compact and nonempty. We will show that there is a continuous bijection 
$$f\colon X \to \lim_{\xleftarrow[n \in \N]{}} G(\script{U}_n).$$

For $x \in X$, we define $f(x) = (q_1(x), q_2(x), \ldots)$. By the product topology, this is a continuous map into the product $\prod_n G(\script{U}_n)$, as all coordinate maps $q_n$ are continuous. Moreover, it is straightforward from the definition of $f_n$ to check that $f(x) \in \lim\limits_{\longleftarrow} G(\mathcal{U}_n)$. That the map $f$ is surjective follows from the fact that each $q_n$ is continuous and $X$ is compact (see \cite[2.22]{Nadler}). Finally, $f$ is injective because of the neighborhood bases requirement (b) on $\mathcal{U}_n$.  Since $X$ is compact and $\lim\limits_{\longleftarrow} G(\mathcal{U}_n)$ is Hausdorff, it follows that $f$ is a homeomorphism as desired, and properties (\ref{connectedness}) and (\ref{item:evencuts}) now follow from Proposition \ref{prop:invlimhyp}.

For the moreover part, simply require that besides (a) and (b), our sequence of multi-cuts $\set{\script{U}_n}:{n \in \N}$ also satisfies
\begin{enumerate}
\item[(c)] every multi-cut $\script{U}_n$ partitions $V(X)$ into regions. 
\end{enumerate}
That this is possible follows from Lemma~\ref{basisregion}; and clearly, property (c) implies that each $f_n$ as defined above will be a monotone map.
\end{proof}

In fact,  a converse of the above theorem holds. This has been mentioned, for completely regular continua, by Nikiel, \cite[3.10(i)]{inverselimitsnikiel}, though without proof. We provide the proof in the language of graph-like continua.

\begin{theorem}
Let $X$ be a countable inverse limit of connected multi-graphs $X_n$ with finite vertex sets $V(X_n)$ and onto monotone bonding maps $f_n \colon X_{n+1} \to X_n$ satisfying: $(+)$ $f_n(V(X_{n+1})) \subseteq V(X_{n})$. 
Then $X$ is a graph-like continuum.
\label{gralikeasinverse}
\end{theorem}

\begin{proof}
By Theorem~\ref{chargraphlike}, every regular continuum with the property that all but a closed zero-dimensional subset of points are of order $2$ is a graph-like continuum. 

That $X$ is regular follows from \cite[3.6]{inverselimitsnikiel}. For sake of completeness, we provide  the argument. Let $\pi_n \colon X \to X_n$ denote the projection maps, and for $m\geq n$ write $f_{m,n} = f_n \circ f_{n+1} \circ \cdots \circ f_{m-1}\circ f_m \colon X_{m+1} \to X_n$.

\emph{Claim: For every $n \in \N$, the set $P_n = \set{y \in X_n}:{\cardinality{\pi_n^{-1}(y)} > 1}$ is countable.}

This holds, because for every $m\geq n$, the set $Q_m = \set{y \in X_n}:{\cardinality{f_{m,n}^{-1}(y)} > 1}$ is countable: By assumption, all bonding maps $f_m$ are monotone, and hence so is $f_{m,n}$. Thus, the collection of non-degenerate $f_{m,n}^{-1}(y)$ from a disjoint collections of subcontinua of $X_{m}$, all with non-empty interior. It follows that $P_n = \bigcup_{m \geq n} Q_m$ is countable, completing the proof of the claim.

To conclude that $X$ is regular, let $x \in X$ and let $U$ be an open neighborhood of $x \in X$. Then there is $k\in \N$ and an open subset $W \subset X_k$ with $x \in \pi_k^{-1}(W) \subset U$. Note that since $X_k$ is a graph, and $P_k$ is countable by the claim, we may choose $W$ with finite boundary such that $\bd{W} \cap P_k = \emptyset$. It follows that $\pi_k^{-1}(W)$ has finite boundary, as well. 

Our candidate for the vertex set of $X$ is
$V(X)= \bigcap_{n \in \N} \pi_n^{-1}(V(X_n))$. 
By $(+)$, the family $\set{\p{V(X_n),f_n}}:{n \in \N}$ gives a well-defined inverse limit, which is identical with our vertex set, i.e.\
$V(X) = \lim_{\leftarrow} \Set{V(X_n),f_n}$.
Since all $V(X_n)$ are finite discrete sets, it follows that $V(X)$ is a compact zero-dimensional metric space, as desired.

To see that elements $y \in X \setminus V(X)$ have order $2$, note that $y \notin V(X)$ means there is an index $N \in \N$ such that $\pi_n(y)$ is an interior point of an edge of $X_n$ for all $n \geq N$. Consider an open neighborhood $U$ with $y \in U \subset X$. As before, there is an index $k> N$ and an open subset $W \subset X_k$ with $y \in \pi_k^{-1}(W) \subset U$. Since $\pi_k(y) \in X_k$ is an interior point of an edge, and $P_k$ is countable by the claim, we may assume that $W$ has 2-point boundary with $\bd{W} \cap P_k =\emptyset$. It follows that $\pi_k^{-1}(W)$ has a 2-point boundary, as well.
\end{proof}

In fact, the class of continua, which can be represented as countable monotone inverse limits of finite connected multi-graphs are precisely the so-called \emph{totally regular continua}, \cite{totallyregular} --  for each countable  $P\subset X$, there is a basis $\script{B}$ of open sets for $X$ so that for each $B \in \script{B}$, $P\cap \bd{B} = \emptyset$ and $B$ has finite boundary. These continua have also been studied under the name \emph{continua of finite degree}.
The class of totally regular continua is strictly larger than the class of completely regular continua. In particular, the condition in Theorem~\ref{gralikeasinverse} on $f_n$ having nontrivial fibers only at vertices cannot be omitted. For example, the universal dendrite $D_n$ of order $n$ can be obtained as the inverse limit of finite connected graphs, see \cite[Section 3]{charatonik}, and $D_n$ has a dense set of points of order $\neq 2$.

 In \cite{graphlikeplanar} the graph-like continuum depicted on the left side of the diagram below served to show that graph-like continua form a wider class than Freudenthal compactifications of locally finite graphs. 
Note that the two black nodes simultaneously act as ends for the blue double ladder, and as vertices for the red edge.

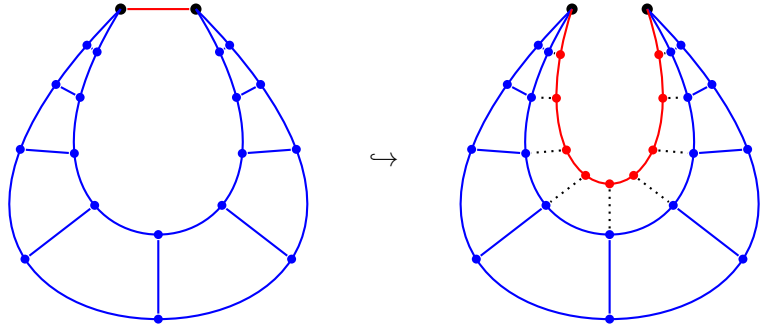
\begin{center}
\begin{tikzpicture}[thick]

\clip (-1,-4.25) rectangle (9,0.25);

\draw (0.5,0) node[circle, fill=black, inner sep=1.5] (v1){};
\draw (1.5,0) node[circle, fill=black, inner sep=1.5] (v2){};

\draw[blue] (0.5,0) .. controls (-2,-4) and (4,-4) .. (1.5,0)
node[circle,pos=0.05,fill=blue,inner sep=1.2] (p1) {}
node[circle,pos=0.11,fill=blue,inner sep=1.2] (p2) {}
node[circle,pos=0.2,fill=blue,inner sep=1.2] (p3) {}
node[circle,pos=0.32,fill=blue,inner sep=1.2] (p4) {}
node[circle,pos=0.5,fill=blue,inner sep=1.2] (p5) {}
node[circle,pos=0.68,fill=blue,inner sep=1.2] (p6) {}
node[circle,pos=0.8,fill=blue,inner sep=1.2] (p7) {}
node[circle,pos=0.89,fill=blue,inner sep=1.2] (p8) {}
node[circle,pos=0.95,fill=blue,inner sep=1.2] (p9) {};

\draw[blue] (0.5,0) .. controls (-5,-5.5) and (7,-5.5) .. (1.5,0)
node[circle,pos=0.03,fill=blue,inner sep=1.2] (n1) {}
node[circle,pos=0.065,fill=blue,inner sep=1.2] (n2) {}
node[circle,pos=0.13,fill=blue,inner sep=1.2] (n3) {}
node[circle,pos=0.28,fill=blue,inner sep=1.2] (n4) {}
node[circle,pos=0.5,fill=blue,inner sep=1.2] (n5) {}
node[circle,pos=0.72, fill=blue,inner sep=1.2] (n6) {}
node[circle,pos=0.87,fill=blue,inner sep=1.2] (n7) {}
node[circle,pos=0.935,fill=blue,inner sep=1.2] (n8) {}
node[circle,pos=0.97,fill=blue,inner sep=1.2] (n9) {};

\draw[red] (v1) -- (v2);
\draw[blue] (p1) -- (n1);
\draw[blue] (p2) -- (n2);
\draw[blue] (p3) -- (n3);
\draw[blue] (p4) -- (n4);
\draw[blue] (p5) -- (n5);
\draw[blue] (p6) -- (n6);
\draw[blue] (p7) -- (n7);
\draw[blue] (p8) -- (n8);
\draw[blue] (p9) -- (n9);

\draw (6.5,0) node[circle, fill=black, inner sep=1.5] (v1){};
\draw (7.5,0) node[circle, fill=black, inner sep=1.5] (v2){};

\draw[blue] (6.5,0) .. controls (4,-4) and (10,-4) .. (7.5,0)
node[circle,pos=0.05,fill=blue,inner sep=1.2] (p1) {}
node[circle,pos=0.11,fill=blue,inner sep=1.2] (p2) {}
node[circle,pos=0.2,fill=blue,inner sep=1.2] (p3) {}
node[circle,pos=0.32,fill=blue,inner sep=1.2] (p4) {}
node[circle,pos=0.5,fill=blue,inner sep=1.2] (p5) {}
node[circle,pos=0.68,fill=blue,inner sep=1.2] (p6) {}
node[circle,pos=0.8,fill=blue,inner sep=1.2] (p7) {}
node[circle,pos=0.89,fill=blue,inner sep=1.2] (p8) {}
node[circle,pos=0.95,fill=blue,inner sep=1.2] (p9) {};

\draw[blue] (6.5,0) .. controls (1,-5.5) and (13,-5.5) .. (7.5,0)
node[circle,pos=0.03,fill=blue,inner sep=1.2] (n1) {}
node[circle,pos=0.065,fill=blue,inner sep=1.2] (n2) {}
node[circle,pos=0.13,fill=blue,inner sep=1.2] (n3) {}
node[circle,pos=0.28,fill=blue,inner sep=1.2] (n4) {}
node[circle,pos=0.5,fill=blue,inner sep=1.2] (n5) {}
node[circle,pos=0.72, fill=blue,inner sep=1.2] (n6) {}
node[circle,pos=0.87,fill=blue,inner sep=1.2] (n7) {}
node[circle,pos=0.935,fill=blue,inner sep=1.2] (n8) {}
node[circle,pos=0.97,fill=blue,inner sep=1.2] (n9) {};

\draw[red] (6.5,0) .. controls (5.5,-3.1) and (8.5,-3.1) .. (7.5,0)
node[circle,pos=0.07,fill=red,inner sep=1.2] (q1) {}
node[circle,pos=0.15,fill=red,inner sep=1.2] (q2) {}
node[circle,pos=0.28,fill=red,inner sep=1.2] (q3) {}
node[circle,pos=0.39,fill=red,inner sep=1.2] (q4) {}
node[circle,pos=0.5,fill=red,inner sep=1.2] (q5) {}
node[circle,pos=0.61, fill=red,inner sep=1.2] (q6) {}
node[circle,pos=0.72,fill=red,inner sep=1.2] (q7) {}
node[circle,pos=0.85,fill=red,inner sep=1.2] (q8) {}
node[circle,pos=0.93,fill=red,inner sep=1.2] (q9) {};

\draw[blue] (p1) -- (n1);
\draw[blue] (p2) -- (n2);
\draw[blue] (p3) -- (n3);
\draw[blue] (p4) -- (n4);
\draw[blue] (p5) -- (n5);
\draw[blue] (p6) -- (n6);
\draw[blue] (p7) -- (n7);
\draw[blue] (p8) -- (n8);
\draw[blue] (p9) -- (n9);

\draw[black,dotted] (q1) -- (p1);
\draw[black,dotted] (q2) -- (p2);
\draw[black,dotted] (q3) -- (p3);
\draw[black,dotted] (q4) -- (p4);
\draw[black,dotted] (q5) -- (p5);
\draw[black,dotted] (q6) -- (p6);
\draw[black,dotted] (q7) -- (p7);
\draw[black,dotted] (q8) -- (p8);
\draw[black,dotted] (q9) -- (p9);

\node at (4,-2) {$\hookrightarrow$};
\end{tikzpicture}
\end{center}


However, after subdividing the red edge appropriately -- turning it into a double ray -- we see from the right side that it can be realized as a standard subspace of the Freudenthal compactification of the triple ladder.
We now show that every graph-like continuum has the same property.
\begin{theorem}
\label{standardsubspacetheorem}
Every graph-like continuum can be embedded as a standard subspace of a Freudenthal compactification of a locally finite graph.
\end{theorem}
We remark that Theorem~\ref{standardsubspacetheorem} can  be rephrased as saying that \emph{every graph-like continuum has a subdivision, turning each edge into a double ray, which is a standard subspace of a Freudenthal compactification of a locally finite graph.}

In the proof of Theorem~\ref{standardsubspacetheorem}, we use the following notation. Let $G$ be a finite, connected graph with vertex set $V$, and let $L(G)$ be its (connected) line graph, both considered as 1-complexes. For every edge $e \subset G$, let $m_e \in e$ be the mid-point of that edge. Then by $G^\circledast$ we denote the graph 
$$G^{\circledast} = \p{G \oplus L(G)} /_\sim, \; \text{ where } m_e \sim e \text{ for } m_e \in G \text{ and } e \in V(L).$$
Geometrically, we subdivide each edge of $G$ in its mid-point, and connect two new such vertices if and only if their underlying edges share a common vertex.

\begin{proof}[Proof of Theorem~\ref{standardsubspacetheorem}]
Let $X$ be a graph-like continuum. Represent $X$ as a monotone inverse limit of finite multi-graphs $G_n$ with onto, monotone simplicial bonding maps $f_n \colon G_{n+1} \to G_n$ having non-trivial fibres at vertices only. 

Recall first that the Freudenthal compactification of a locally finite graph can be realized as an inverse limit: Let $L$ be a locally finite graph with vertex set $V(L)= \set{v_k}:{k \in \N}$ say. Let $k_n$ be an increasing sequence of integers, and consider for each $n$ the induced subgraph $L_n = L[v_0, \ldots, v_{k_n}]$. Let $L^n$ denote the multi-graph quotient of $L$ where we contract every connected component of the induced subgraph $L[V(L) \setminus V(L_n)]$, deleting all arising loops. Since $L$ was locally finite, it is easy to check that $L^n$ is a finite multi-graph. Then $\set{L^n}:{n \in \N}$ forms an inverse system under that natural projection maps $g_n \colon L^{n+1} \to L^n$, such that the resulting inverse limit $\lim_{\leftarrow} L^n \cong \gamma L$
is the Freudenthal compactification of $L$; moreover, this holds independently of the sequence $k_n$. 

Now our proof strategy is as follows. We plan to find a locally finite graph $L$ as above such that there are subgraphs $T_n \subset L^n$ such that
\begin{enumerate}[label=(\roman*)]
	\item\label{subsystem} $\hat{g}_n = g_n \restriction V(T_{n+1}) \to V(T_n)$ restricts to a surjection (so that the $T_n$ form a subsystem of the inverse limit with bonding maps $\hat{g}_n$), and 
	\item\label{topminors} for each $n \in \N$, the graph $T_n$ witnesses that $G_n$ is a topological minor of $L^n$, meaning there are homeomorphisms $h_n \colon G_n \to T_n$ of the underlying 1-complexes which map $V(G_n) \hookrightarrow V(T_n)$, and map distinct edges $vw$ and $xy$ of $G_n$ to independent $h_n(v)h_n(w)$- and $h_n(x)h_n(y)$-paths in $T_n$, and
    \item\label{commute} we have $\hat{g}_n \circ h_{n+1} =  h_n \circ f_n$ for all $n$, i.e.\ the following diagram commutes: \begin{center}
\begin{tikzpicture}
  \matrix (m)
    [
      matrix of math nodes,
      row sep    = 3em,
      column sep = 4em
    ]
    {
      T_n             & T_{n+1} \\
      G_n  &  G_{n+1}           \\
    };
  \path
    (m-2-1) edge [->] node [left] {$h_{n}$} (m-1-1)
    (m-1-1.east |- m-1-2)
      edge [<-] node [above] {$\hat{g}_n$} (m-1-2);
        \path
     (m-2-1) edge [<-] node [above] {$f_n$} (m-2-2);
        \path
     (m-2-2) edge [->] node [right] {$h_{n+1}$} (m-1-2);
\end{tikzpicture}
\end{center}
\end{enumerate}
Under these assumptions, it follows that $X = \lim_{\leftarrow} G_n$ is homeomorphic to the inverse limit $\lim_{\leftarrow} T_n$, which in turn, as it was constructed as a subsystem, embeds into the inverse limit $\lim_{\leftarrow} L^n = \gamma L$, which equals the Freudenthal compactification of $L$ by the foregoing discussion. Thus, it remains to find a locally finite graph $L$ subject to requirements \ref{subsystem}--\ref{commute}.

We will build this locally finite graph $L$ by geometric considerations as a direct limit of finite connected graphs $F_n$, so that $F_n = L[V(F_n)] = L_n$. More precisely, we will define finite connected 1-complexes $F_n$ such that 
\begin{enumerate}
\item\label{directlimit} $F_0 \hookrightarrow F_1 \hookrightarrow F_2 \hookrightarrow \cdots$ forms a direct limit such that for all $n>0$, no vertex of $F_{n+1} \setminus F_{n}$ is incident with a vertex of $F_{n-1}$, and 
\item\label{linegraph} $F_n$ is embedded together with $G_n$ in some ambient 1-complex $H_n=F_n \cup G_n$ such that
	\begin{enumerate}[label=(\alph*)]
  	 	\item no vertex of $G_n$ lies in $F_n$,
        \item every vertex of $F_n$ lies on an edge of $G_n$,
        \item every open edge of $F_n$ is either disjoint from $G_n$, or completely contained in an edge of $G_n$, and 
    	\item every edge of $G_n$ intersects with $F_n$ in a non-trivial path $P \subset F_n$ such that the end-vertices of $P$ are vertices of $V(F_n) \setminus V(F_{n-1})$.
   \end{enumerate}
\end{enumerate}
To begin, put $H_0=G_0^{\circledast}$, and let $F_0$ denote the subgraph $L(G_0) \subset G_0^{\circledast}$. Then (\ref{linegraph}) is satisfied since vertices of $F_0$ are mid-points of edges of $G_0$, and every open edge of $F_0$ is disjoint from $G_0$; and (\ref{directlimit}) is trivially true.

\begin{figure}[ht]

\begin{tikzpicture}

 \node[draw, circle, scale=1] (x) at (0,0) {$x$};
 \node[draw, circle, scale=1] (y) at (3,0) {$y$};
 \node[draw, circle, scale=1] (z) at (0,3) {$z$};
 \node[draw, circle, scale=1] (v) at (3,3) {$v$}; 

	\draw[thick] (x) -- (y);
    \draw[thick] (x)-- (z);
    \draw[thick] (z)-- (v);
    \draw[thick] (y)-- (v);
    
  \node[draw, circle, red, fill=red, inner sep=0.4pt,minimum width=4pt] (e_1) at (0,1.5) {$$}; 
  \node[draw, circle, red, fill=red, inner sep=0.4pt,minimum width=4pt] (e_2) at (1.5,0) {$$};
  \node[draw, circle, red, fill=red, inner sep=0.4pt,minimum width=4pt] (e_3) at (3,1.5) {$$};
  \node[draw, circle, red, fill=red, inner sep=0.4pt,minimum width=4pt] (e_4) at (1.5,3) {$$};
 \draw [thick,red, bend left,] (e_1) to (e_2);   
  \draw [thick,red, bend left,] (e_2) to (e_3);   
   \draw [thick,red, bend left,] (e_3) to (e_4);   
    \draw [thick,red, bend left,] (e_4) to (e_1);   

\node (inversemap) at (4,1.5) {$\xleftarrow{f_{0}}$};

 \node[draw, circle, scale=1] (x1) at (5,0) {$x$};
 \node[draw, circle, scale=1] (y1) at (8,0) {$y$};
 \node[draw, circle, scale=1] (z1) at (5,3) {$z$};

\node[draw, circle, scale=.7] (v11) at (7.3,3) {$v_1$}; 
\node[draw, circle, scale=.7] (v12) at (8,2.3) {$v_2$}; 
\node[draw, circle, scale=.7] (v13) at (8.2,3.2) {$v_3$}; 

	\draw[thick] (x1) -- (y1);
    \draw[thick] (x1)-- (z1);
    \draw[thick] (z1)-- (v11);
    \draw[thick] (y1)-- (v12);
    \draw[thick] (v11)-- (v12);
    \draw[thick] (v11)-- (v13);
    \draw[thick] (v12)-- (v13);
    
      \node[draw, circle, red, fill=red, inner sep=0.4pt,minimum width=4pt] (ee_1) at (5,1.5) {$$}; 
  \node[draw, circle, red, fill=red, inner sep=0.4pt,minimum width=4pt] (ee_2) at (6.5,0) {$$};
  \node[draw, circle, red, fill=red, inner sep=0.4pt,minimum width=4pt] (ee_3) at (8,1.15) {$$};
  \node[draw, circle, red, fill=red, inner sep=0.4pt,minimum width=4pt] (ee_4) at (6.15,3) {$$};
 \draw [thick,red, bend left,] (ee_1) to (ee_2);   
  \draw [thick,red, bend left,] (ee_2) to (ee_3);   
   \draw [thick,red, bend left,] (ee_3) to (ee_4);   
    \draw [thick,red, bend left,] (ee_4) to (ee_1); 
    
     \node[draw, circle, blue, fill=blue, inner sep=0.4pt,minimum width=4pt] (e_5) at (6.675,3) {$$}; 
  \node[draw, circle, blue, fill=blue, inner sep=0.4pt,minimum width=4pt] (e_6) at (8,1.675) {$$};
  \node[draw, circle, blue, fill=blue, inner sep=0.4pt,minimum width=4pt] (e_7) at (7.75,3.095) {$$};
  \node[draw, circle, blue, fill=blue, inner sep=0.4pt,minimum width=4pt] (e_8) at (7.65,2.65) {$$};
    \node[draw, circle, blue, fill=blue, inner sep=0.4pt,minimum width=4pt] (e_9) at (8.1,2.75) {$$};
    
     \draw [thick,blue, bend left,] (e_6) to (e_5);   
     
       \draw [thick,blue, bend left,] (e_6) to (e_8);  
         \draw [thick,blue, bend left,] (e_8) to (e_5);  
         
         \draw [thick,blue, bend left=60] (e_9) to (e_6);  
         \draw [thick,blue, bend left=60] (e_5) to (e_7); 
         
         \draw[thick, blue, bend left] (e_7) -- (e_8);
  		  \draw[thick, blue, bend left] (e_8)-- (e_9);
      	\draw[thick, blue, bend left] (e_9)-- (e_7);

\node[draw, circle, blue, fill=blue, inner sep=0.4pt,minimum width=4pt] (e_10) at (5.675,3) {$$};  
\node[draw, circle, blue, fill=blue, inner sep=0.4pt,minimum width=4pt] (e_11) at (5,2.25) {$$};  
        \draw [thick,blue, bend left,] (e_10) to (e_11); 
        
        \node[draw, circle, blue, fill=blue, inner sep=0.4pt,minimum width=4pt] (e_12) at (5,0.75) {$$};  
\node[draw, circle, blue, fill=blue, inner sep=0.4pt,minimum width=4pt] (e_13) at (5.75,0) {$$};  
        \draw [thick,blue, bend left,] (e_12) to (e_13); 
        
\node[draw, circle, blue, fill=blue, inner sep=0.4pt,minimum width=4pt] (e_14) at (7.25,0) {$$};  
\node[draw, circle, blue, fill=blue, inner sep=0.4pt,minimum width=4pt] (e_15) at (8,0.7) {$$};  
        \draw [thick,blue, bend left,] (e_14) to (e_15);

        	\draw[very thick, ForestGreen, bend left] (ee_4)-- (e_5);
 	     	\draw[very thick, ForestGreen, left] (ee_3)-- (e_6);       
        	\draw[very thick, ForestGreen, left] (ee_4)-- (e_10);
            \draw[very thick, ForestGreen, left] (ee_1)-- (e_11);
            \draw[very thick, ForestGreen, left] (ee_1)-- (e_12);
            \draw[very thick, ForestGreen, left] (ee_2)-- (e_13);
            \draw[very thick, ForestGreen, left] (ee_2)-- (e_14);
            \draw[very thick, ForestGreen, left] (ee_3)-- (e_15);

\end{tikzpicture}
\caption{Depicts the first bonding map $f_0$ between graphs $G_1$ and $G_0$ in black, where $f\p{\Set{v_1,v_2,v_3}} = \singleton{v}$. Further, the figure on the left shows $F_0 \subset G_0^{\circledast}$ in red, and on the right $F_1 \subset H_1$ as the union of $\tilde{F}_0$ in red, $\bigcup_v L_v$ in blue, and edges induced by $\tilde{F}_0$ and $\bigcup_v L_v$ in green.}
\label{freundenthal}
\end{figure}
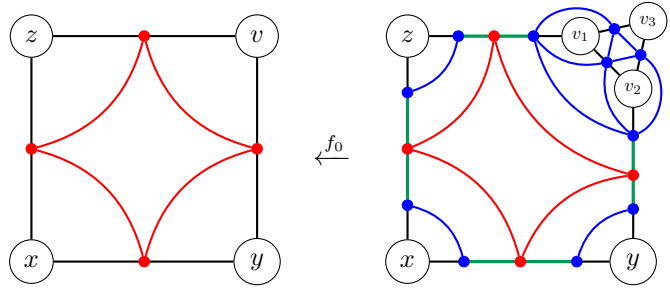

Now inductively, suppose we have already defined $H_n = G_n \cup F_n$ for some $n \in \N$ according to (\ref{directlimit}) and (\ref{linegraph}). First, consider the natural pull-back $\tilde{F}_n \subset G_{n+1}$ of $F_n$ under $f_n$. More precisely, by (\ref{linegraph}), the preimage $f_n^{-1}(F_n) \subset G_{n+1}$ is isomorphic to a subgraph of $F_n$. Let $\tilde{F}_n$ be an isomorphic copy of $F_n$ on the vertex set $f^{-1}_n(V(F_n))$ obtained by adding all edges missing from $f_n^{-1}(F_n)$ so that they are disjoint from $G_{n+1}$.

For every component $C_v$ of the topological subspace $H_n \setminus F_n$ (which by (\ref{linegraph})(a) and (b) will be a vertex $v$ of $G_n$ incident with finitely many half-open edges), consider the subcontinuum $K_v=\closure{f_n^{-1}(C_v)} \subset G_{n+1}$. Then $K_v$ is a finite connected graph. For each $v$, consider 
$K_v^{\circledast}, \; \text{ and } \; L_v =L(K_v) \subset K_v^{\circledast}$, 
and define $F_{n+1}$ to be the induced subgraph $F_{n+1} = \tilde{F}_n \cup \bigcup \{ L_v : v \in V(G_n)\}$.

{\bf Claim 1: }{\it $F_{n+1}$ is a connected graph.}

By induction on $n$. If $F_n$ is connected, then so is its isomorphic copy $\tilde{F}_n$. As line graphs of connected graphs, every $L_v$ is connected. Since by construction, every $L_v$ is connected via an induced edge to $\tilde{F}_n$, it follows that $F_{n+1}$ is connected.

{\bf Claim 2: }{\it Property (\ref{linegraph}) holds for $F_{n+1}$ and $G_{n+1}$.}

(a) No vertex $v$ of $G_{n+1}$ lies on $\tilde{F}_n$, as otherwise $f_n(v)$ would be a vertex of $G_n$ on $F_n$. Also, since all $L_v$ are partial line graphs of $G_{n+1}$, we see that (a) holds at step $n+1$.

(b) Similar.

(c) By construction, this holds for all edges of $\tilde{F}_n$. Further, all edges of $L_v$ are disjoint from $G_{n+1}$, and all edges of $F_{n+1}$ induced $\tilde{F}_n$ and $L_v$ are completely contained in one edge of $G_{n+1}$ be definition.

(d) Let $e=vw$ be an edge of $G_{n+1}$. If $e \notin E(G_n)$ then $F_{n+1} \cap e = L_v \cap e$ is a trivial path consisting of one new vertex. Otherwise, if $e \in E(G_n)$, then by construction and induction assumption, $\tilde{F}_n$ intersects $e$ in a non-trivial path $P \subset \tilde{F}_n$ such that the end-vertices of $P$ have been added only at the previous step. But now, we see that $F_{n+1} \cap G$ is a path $P'$ which is one edge longer on either side than $P$, because we added two edges induced by $L_v$ and $L_w$. In particular, the end vertices of $P'$ are vertices of $L_v$ and $K_v$, and so have only been added at this step.

{\bf Claim 3: }{\it Property (\ref{directlimit}) holds.}

Since $F_n \cong \tilde{F}_n \subset F_{n+1}$ it is clear how to choose the embedding $F_n \hookrightarrow F_{n+1}$. The second part of the claim now follows from (\ref{linegraph})(d) as follows: Every vertex of $F_{n+1} \setminus F_n$ is a vertex of some $L_v$. By construction, any such vertex is connected at most to one of the end vertices on some path $P$, which is, by (\ref{linegraph})(d), a vertex of $F_{n} \setminus F_{n-1}$. 

This completes the recursive construction. As indicated above, the graphs $F_0 \hookrightarrow F_1 \hookrightarrow F_2 \hookrightarrow \cdots$ give rise to a direct limit, which we call $L$. Let $V(L)=\set{v_k}:{k \in \N}$ be an enumeration of the vertices of $L$, first listing all vertices of $F_0$, then all (remaining) vertices of $F_1$ etc. It is clear that there is an increasing sequence of integers $k_n$ such that $L_n = L[\Set{v_0, \ldots, v_{k_n}}]=F_n$.

{\bf Claim 4: }{\it $L$ is a locally finite connected graph.}

To see that $L$ is locally finite, note that any vertex $v \in L$ is contained in some $F_n$ for some $n$, and then  (\ref{directlimit}) implies that $deg_L(v) = deg_{F_{n+1}}(v) < \infty$. And since every $L_n$ is connected, so is $L$.

{\bf Claim 5: }{\it There are isomorphisms $\varphi_n \colon H_n \to L^n$.}
It suffices to show that $L^n \cong F_{n+1} / \set{L_v}:{v \in V(G_n)}$. Indeed, (\ref{directlimit}) implies that the connected components of $L \setminus F_n$ correspond bijectively to the connected components of $F_{n+1} \setminus F_n$, which are, by construction, precisely the $L_v$ indexed by the different $v \in V(G_n)$. In particular, $\varphi_n$ is a bijection between $V(G_n)$ and the dummy vertices of $L^n$ that commutes with the respective bonding maps, i.e.\ 
$$(\dagger) \quad \quad g_{n-1} \circ \varphi_n(v) \restriction V(G_n) = \varphi_{n-1} \circ f_{n-1}(v) \restriction V(G_n).$$

{\bf Claim 6: }{\it For $T_n=\varphi_n(G_n) \subset L^N$ the subgraph of $L^n$ which is the image of 1-complex $G_n \subset H_n$ subdivided by the vertices of $L_n$, satisfies \ref{subsystem}--\ref{commute}.} Everything is essentially set up by construction; \ref{commute} follows by $(\dagger)$ with $h_n = \varphi_n \restriction G_n$.
\end{proof}

Note that our embedding of $X$ into $\gamma L$ has the property that every vertex of the graph-like continuum $X$ is represented by a compactification point (an end) of $\gamma L$. By exercising some extra care in the above construction, one could arrange for isolated vertices of $V(X)$ to be mapped to vertices of $L$.


\subsubsection*{Remark} Theorem~\ref{standardsubspacetheorem} has the following notable consequence. 
Diestel asked in \cite{dqu} whether every connected subspace of the Freudenthal compactification of a locally finite graph is automatically arc-connected. In 2007, Georgakopoulos gave a negative answer, \cite{agelosconnected}. 
However, the analogous problem for arbitrary continua is a  well-studied problem. Indeed, a continuum is said to be \emph{in class A} if every connected subset is arc-wise connected. Continua in class A have been characterized by Tymchatyn in 1976, \cite{tymchatyn}. Even earlier,  in 1933, Whyburn  gave an example of a completely regular continuum which is not in class A, \cite[Example~4]{Whyburn}. Applying Theorem~\ref{standardsubspacetheorem}, Whyburn's example shows at once that Freudenthal compactifications of locally finite graphs are not necessarily in class A.

\section{Eulerian graph-like continua}\label{EulerPf}

\subsection{Characterizing Eulerian Graph-Like Continua}\label{pfthmab}

We now prove the equivalence of (i), (ii), (iii) and (iv) of Theorem~(B) in the case of \emph{closed} paths, and then deduce the same equivalences in  Theorem~(B) for \emph{open} paths. 
To start note that (iv) $\Rightarrow$ (iii) and (i) $\Rightarrow$ (iii) of Theorem~(B) follow from Lemma~\ref{lemma:evencuts}. The next lemma takes care of  (iii) $\Rightarrow$ (iv).

\begin{lemma}
\label{lemma5}
A graph-like continuum such that every topological cut is even can be decomposed into edge-disjoint topological cycles.
\end{lemma}
\begin{proof}
We adapt the proof from \cite{NW} as follows. Let $G$ be a graph-like continuum, and $E(G)=\{e_0,e_1,\ldots\}$ an enumeration of its edges. Note that $G-e_0$ is not disconnected: If $G-e_0 = A \oplus B$ then $(A,B)$ would be a separation in $G$ with $E(A,B) =\singleton{e}$, so odd, a contradiction. Since $G$ is arc-connected by Corollary~\ref{souslian}, there is an arc in $G-e_0$ connecting $x$ and $y$. Together with $e_0$ that gives a topological circle $C_0$.

Now let $e_i=x_iy_i$ be the first edge not on $C_0$. We claim that there is a path connecting $x_i$ to $y_i$ in $G \setminus (E(C_0)\cup \{e_i\})$. Otherwise, there is a cut $(A,B)$ of $G'=G \setminus E(C_0)$ such that $E_{G'}(A,B)=\{e_i\}$. But then the same cut viewed in $G$ would be odd by Lemma \ref{lemma:evencuts}. A contradiction.

It is clear that we can continue in this fashion until all edges are covered.
\end{proof}

To establish the equivalence of clauses (i), (iii) and (iv) of Theorem~(B), it remains to show (iii) implies (i), which is established by  the next result.

\begin{prop}
\label{theorem16}
Let $X$ be a graph-like continuum.
If all topological cuts of $X$ have even size then $X$ has an Eulerian loop.
\end{prop}

\begin{proof}
By Theorem~\ref{inverselimit} (\ref{item:evencuts}), $X$ can be written as an inverse limit of graphs $G_n$, which are all closed Eulerian. Let $f_n$ denote the bonding map $f_n \colon G_{n+1}\to G_n$.

For each $n$, let $\mathcal{E}_n$ be the collection of all Euler cycles of $G_n$. Since $G_n$ is finite, so is $\mathcal{E}_n$. For each $n\in\mathbb{N}$, let $\hat{f}_n \colon \mathcal{E}_{n+1}\to \mathcal{E}_n$ be the map induced by $f_n$. That is, if $E=(v_0e_0v_1e_1v_2e_2\cdots v_ke_kv_0)$, then \[\hat{f}_n(E)=\p{f_n(v_0)f_n(e_0)f_n(v_1)f_n(e_1)\cdots f_n(e_k)f_n(v_0)}.\] 
Observe that from the proof of Theorem \ref{inverselimit} some of the edges in $E$ get contracted to a vertex. So $\hat{f}_n(E)$ is an Eulerian circuit in $G_n$. Now, $\{\mathcal{E}_n,\hat{f}_n\}_{n\in\mathbb{N}}$ forms an inverse system, and since each $\mathcal{E}_n$ is compact, we see  $\lim_{\leftarrow}\mathcal{E}_n\neq\emptyset$. 

Let $(E_n)\in\lim_{\leftarrow}\mathcal{E}_n$. For each $n \in \N$, fix an Eulerian loop $\phi_n\colon S^1\to G_n$ following the pattern given by $E_n$. 
Now observe that since the $(E_n)_{n \in \N}$ are compatible, there are monotone continuous maps $g_n \colon S^1 \to S^1$ ($n \in \N$) such that the diagram 
\begin{center}
\begin{tabular}{ccccccc}
$G_1$ & $\xleftarrow{f_1}$ & $G_2$ & $\xleftarrow{f_2}$ & $G_3$ & $\xleftarrow{f_3}$ & $\cdots$ \\
$\uparrow\phi_1$ &    & $\uparrow\phi_2$ &  & $\uparrow\phi_3$ &   &  \\ 
$S^1$ & $\xleftarrow{g_1}$ & $S^1$ & $\xleftarrow{g_2}$ & $S^1$ & $\xleftarrow{g_3}$ & $\cdots$
\end{tabular}
\end{center}
commutes. As an inverse limit of circles under monotone bonding maps, we have $\lim\limits_{\longleftarrow}S^1\cong S^1$, \cite[4.11]{Chapel}, and so the map $$g\colon\lim\limits_{\longleftarrow}S^1\to \lim\limits_{\longleftarrow}G_n, \; (x_n)_{n \in \N} \mapsto (\phi_n(x_n))_{n \in \N}$$
is our desired Eulerian loop.
\end{proof}

The proof of the equivalence of (i), (ii), (iii) and (iv) in Theorem~(B), for \emph{closed} loops, is completed by Lemma~\ref{prop:eveniffeven} showing the equivalence of (ii) and (iii). A preliminary lemma is needed.

\begin{lemma}
Let $X$ be a graph-like continuum, $(A,B)$ be a separation of $V,$ and $\mathcal{U}=\{A_1, \dots , A_n\}$ be a multi-cut of $A$. If the cut $E(A,B)$ is odd, then $E(A_j, V\setminus A_j)$ is odd for some $1\leq j \leq n$.
\label{prop:sequenceoddcuts}
\end{lemma}

\begin{proof}
Consider the contraction graph induced by the multi-cut $(B,A_1,\ldots,A_n)$.

By assumption, the vertex $\Set{B}$ has odd degree. Since by the Handshaking Lemma, the number of odd-degree vertices in a finite graph is even, there must be some further vertex $\Set{A_j}$ with odd degree, so $E(A_j, V\setminus A_j)$ is odd.
\end{proof}

\begin{lemma}\label{lem_even}
Let $X$ be a graph-like continuum. All topological cuts of $X$ are even if and only if every vertex of $X$ is even.
\label{prop:eveniffeven}
\end{lemma}

\begin{proof}
If all cuts are even, then from the definition every vertex is even. We prove the converse by contrapositive. Assume there exists a separation $(A_0,B_0)$ of $V$ such that $E(A_0,B_0)$ is odd. Let $\mathcal{U}_0=\{U_{1_0}, \dots , U_{n_0}\}$ be a separation of $A_0$ into sets with diameter $<\frac{1}{2} diam(A_0)$. By Lemma~\ref{prop:sequenceoddcuts} there exists $1\leq j_0\leq n_0$ such that $E(U_{j_0}, V\setminus U_{j_0})$ is odd. Denote $U_{j_0}$ by $A_1$ and $V\setminus U_{j_0}$ by $B_1$. Let $\mathcal{U}_1=\{U_{1_1}, \dots , U_{n_1}\}$ be a separation of $A_1$ into sets with diameter $<\frac{1}{2} diam(A_1)$. Again by Lemma~\ref{prop:sequenceoddcuts} there exists $1\leq j_1\leq n_1$ such that $E(U_{j_1}, V\setminus U_{j_1})$ is odd. Denote $U_{j_1}$ by $A_2$ and $V\setminus U_{j_1}$ by $B_2$. Continuing with this procedure we obtain a nested sequence of nonempty cut elements $\{A_i\}_{i\in\N}$. By construction $\bigcap_{i\in\N}A_i=\{v\}\in V$ and $E(A_i,B_i)$ is odd for every $i\in\N$, hence $v$ is not even.
\end{proof}

It remains to deduce the equivalence of (i), (ii), (iii) and (iv) in  Theorem~(B) for the case of \emph{open} paths from that of \emph{closed} paths. This can be achieved with a simple trick.

Suppose, to start, that item (i) for open paths of Theorem~(B), holds for a graph-like continuum $X$. So in $X$ there is an  open Eulerian path starting at a 
vertex $v$ and ending at another vertex $w$. Create a new graph-like continuum $Z$ by adding one edge to $X$ with endpoints at $v$ and $w$. Then $Z$ is a graph-like continuum  with an Eulerian loop.
So, by Theorem~(B) applied to $Z$, each of (ii)-(iv) (for closed paths) of that theorem hold for $Z$. But now it easily follows from the definitions that each of (ii)-(iv) (for open paths) of Theorem~(B) hold for $X$.

Now let $X$ be a graph-like continuum for which one of items (ii)-(iv) for open paths in Theorem~(B) holds. To complete the deduction we show (i) holds for open paths. Each of these items highlights two distinct vertices (the two odd vertices in (ii) and the ends of the arc in (iv)). Call them $v$ and $w$.
 Create a new graph-like continuum $Z$ by adding one edge to $X$ with endpoints at $v$ and $w$. Then $Z$ is a graph-like continuum and it is easily verified from the definitions that it satisfies one of (ii)-(iv) for closed paths in Theorem~(B).  Hence (i) for closed paths of Theorem~(B) holds, and there is a closed Eulerian path in $Z$.  Removing the added edge yields an open Eulerian path in $X$.

\subsection{Counting All Eulerian Loops and Paths}\label{eulercnt}

In this section we aim to count the number of distinct Eulerian loops and paths in a given graph-like continuum. 
To do so we must decide what it means for two paths to be equivalent. This is a well-studied problem in combinatorial group theory, and we adopt the approach taken there.
Two maps $f,g \colon I \to X$ are \emph{equivalent} if $v=f(0)=g(0)$, $w=f(1)=g(1)$, $v$ and $w$ are vertices, and $f$ is homotopy equivalent to $g$ relative to $v,w$. As noted in the Introduction, every map $f\colon I \to X$ with vertices for endpoints is equivalent to a standard path.

Let $X$ be a graph-like continuum. 
By Theorem~\ref{inverselimit} (\ref{item:evencuts}), $X$ can be written as an inverse limit of graphs $G_n$, via bonding maps $f_n:G_{n+1}\to G_n$.
As in Proposition~\ref{theorem16}, 
for each $n$, let $\mathcal{E}_n$ be the collection of all Eulerian cycles in $G_n$, and let $\hat{f}_n\colon \mathcal{E}_{n+1}\to \mathcal{E}_n$ be the map induced by $f_n$.  
Recall, $(\mathcal{E}_n,\hat{f}_n)_n$ forms an inverse system, and set  $\mathcal{E} = \mathcal{E}(X)=\lim_{\leftarrow}\mathcal{E}_n$. As in Proposition~\ref{theorem16}, every $(E_n)_n$ in $\mathcal{E}(X)$ gives rise to an Eulerian loop in $X$. 
It is straightforward to check that distinct members of $\mathcal{E}(X)$ gives rise to inequivalent Eulerian loops.
The converse is also true, although we do not need that for our counting result. In 
any case we consider $\mathcal{E}(X)$ to be the \emph{space of Eulerian loops} in $X$.

\begin{theorem}
\label{finiteorcontinuum}
A closed Eulerian graph-like continuum has either finitely many distinct Eulerian loops in which case it is a graph, or it has continuum many Eulerian loops.
\end{theorem}
\begin{proof}
Since every $\script{E}(G_n)$ is finite discrete, the inverse limit is a compact subspace of a Cantor set. As compact subspaces of a Cantor set without isolated points have size continuum, the result follows from the next claim.

{\bf Claim:} If $\script{E}(X)$ contains an isolated point, then $X$ is homeomorphic to a graph.

Fix an isolated element $(E_n)_n$ in $\mathcal{E}(X)$. Fix an Eulerian loop $f \colon I \to X$ of $X$ corresponding to $(E_n)_n$ (as in Proposition~\ref{theorem16}). To witness that $f$ is isolated, find coordinate graph $G_n$ induced by a multi-cut $\script{U}=(U_1,\ldots,U_n)$ of $X$ such that the  the quotient map $q \colon X \to  G$ acting on the set of (distinct) Euler cycles $\script{E}(X) \to \script{E}(G_n)$ satisfies $q^{-1}(q(f)) = \singleton{f}$.
We claim that every $X[U_i]$ (the subspace of $X$ induced by the vertex set $U_i$) is a graph. This would show that  $X$ itself is also a graph.

Without loss of generality,  $f(0) \notin X[U_i]$. The map $f$ induces a linear order on $E(U_i, V \setminus U_i)$, say $(e_0,\ldots,e_{2k-1})$. For all $0 \leq l < 2k$ write $x_l$ for the end vertex of $e_l$ in $U_i$ (of course, the $x_l$ need not be distinct). Let $f_m$ be the arc between $x_{2m}$ and $x_{2m+1}$ induced by $f$. We claim that the arcs $\set{f_m}:{0 \leq m < k}$ witness that $X[U_i]$ is a graph. 

First of all, $X[U_i]=\bigcup_{m < k} f_m$ since $f(0) \notin X[U_i]$ implies $f_m \subseteq  X[U_i]$, and $f$ Eulerian implies that all edges in $E(U_i,U_i)$ are hit. As the edges are dense, all of $X[U_i]$ is covered.

To complete the proof, it remains to show that our arcs intersect pairwise only finitely. Indeed, we claim that $|\mathring{f}_m \cap \mathring{f}_p| \leq 1$. Otherwise, suppose that $y\neq z$ are two vertices lying in the interior of both arcs.
Denote by $e_m=f_m \restriction [y,z]$ and $e_p=f_p \restriction [y,z]$ (or $e_m=f_m \restriction [z,y]$ depending on which vertex comes first). Since $f_m,f_p$ are edge disjoint, $e_m \neq e_p$. Then replace 
\begin{itemize}
\item $f_m$ by $f_m \restriction [x_{2m},y] \cup e_p \cup f_m \restriction [y,x_{2m+1}] $, and 
\item $f_p$ by $f_p \restriction [x_{2p},y] \cup e_m \cup f_p \restriction [y,x_{2p+1}] $.
\end{itemize}
This change gives rise to an Eulerian loop $f'$ of $X$ distinct from $f$, with $q^{-1}(q(f)) \supset \Set{f,f'}$, a contradiction.
\end{proof}

We can deduce the analogous result for the number of open Eulerian paths by the same trick used to derive the open version of Theorem~(B) from the closed version. Let $X$ be an open Eulerian graph-like continuum, and let $v,w$ be the two odd vertices of $X$.   Add an edge connecting  $v$ and $w$, to get a closed Eulerian graph-like continuum $Z$. Apply the preceding result to deduce $Z$ has either finitely many distinct Eulerian loops in which case it is a graph, or it has continuum many Eulerian loops. Removing the added edge yields either that $X$ is a graph or has continuum many open Eulerian paths.
\begin{theorem}
\label{finiteorcontinuumOpen}
An open Eulerian graph-like continuum has either finitely many distinct open Eulerian paths in which case it is a graph, or it has continuum many open Eulerian paths.
\end{theorem}

\section{Bruhn \& Stein Parity}\label{sec:BruhnStein}

Let $X$ be a graph-like continuum with vertex set $V$. Let $v$ be a vertex of $X$. Then we say that $v$ has \emph{strongly even degree} (respectively, \emph{strongly odd degree}) if there is a clopen neighborhood $C$ of $v$ such that for every clopen neighborhood $A$ of $v$ contained in $C$ the maximal number of edge-disjoint arcs from $V \setminus A$  to $v$ is even (respectively, odd).
By Lemma~\ref{finitecuts}, this is well-defined. We further say that $v$ has \emph{weakly even degree} (resp., \emph{weakly odd degree}) if $v$ does \emph{not} have strongly odd (resp. even) degree.  
Equivalently, $v$ has weakly even degree if $v$ has a neighborhood base of clopen sets, $C$, so that  the maximal number of edge-disjoint arcs from $V \setminus C$ to $v$ is even. And similarly for weakly odd degree.
Bruhn \& Stein \cite{euler} use the same terminology for `strongly odd' and `weakly even' degrees, but use `even' for our `strongly even' and `odd' for our `weakly odd'.

 Note that isolated vertices have finite degree by Lemma~\ref{finitecuts}, so for them being even and having strongly even degree coincide (and similarly for odd).   
In general, our notion of `even' and `odd'  vertices implies those of Bruhn \& Stein. To see this, we shall need a version of Menger's theorem in the edge-disjoint version. That Menger-like theorems hold for graph-like continua is not surprising, and vertex-disjoint versions of Menger have been proved in \cite{thomassenvella}. We complement their results by the following theorem. Note that in finite graph theory, the edge disjoint version follows from the vertex disjoint version by applying the latter theorem to the line graph. As it is unclear, what a line-graph for graph-like spaces should be, we need a different proof.
\begin{theorem}[Menger for Graph-like Continua---Edge Disjoint Version]
\label{thm_menger}
Let $X$ be a graph-like continuum. For disjoint closed sets $A$ and $B$ of vertices of $X$, the maximum number of edge-disjoint $A-B$ paths equals the minimum cut separating $A$ from $B$.
\end{theorem}

\begin{proof}
Let $k$ be the size of a smallest cut separating $A$ from $B$. Note that since $A$ and $B$ are closed disjoint, it follows from compactness that such a cut exists, and hence $k$ is finite by Lemma~\ref{finitecuts}. It is clear that the maximum number of edge-disjoint $A-B$ paths is bounded by $k$. 

Conversely, write $X$ as an inverse limit $X = \lim_{\leftarrow} G_n$ with simplicial bonding maps $f_n \colon G_{n+1} \to G_n$ and simplicial projection maps $\pi_n \colon X \to G_n$. Without loss of generality, $\pi_n(A) \cap \pi_n(B) = \emptyset$ for all $n$. Let $\script{T}_n$ be the (finite) space of all $k$-tuples of edge-disjoint connected subgraphs of $G_n$ that intersect both $\pi_n(A)$ and $\pi_n(B)$. By Menger's theorem for finite graphs, $\script{T}_n \neq \emptyset$ for all $n$, so $\script{T}_n$ with natural bonding maps $\hat{f}_n$ form their own inverse system, which is non-empty. Taking the inverse limit in each coordinate, we obtain $k$ edge-disjoint subcontinua of $X$ each intersecting both $A$ and $B$. By Corollary~\ref{souslian}, we can find $A-B$ paths inside each subcontinuum, which are then edge-disjoint by construction. 
\end{proof}

\begin{lemma}\label{lem_deg}
Let $X$ be a graph-like continuum and $v$ an even (resp.  odd) vertex in $X$. Then $v$ has strongly even (resp. odd) degree.
\end{lemma}
\begin{proof}
Let $v$ be an even vertex and let $C$ be a clopen neighborhood of $v$ such that if $A$ is a clopen neighborhood of $v$ contained in $C$, then $E(V(X)\setminus A, A)$ is even. Observe that $E(V(X)\setminus A,A)$ is the minimum cut separating $V(X)\setminus A$ from $v$. Hence by  Theorem~\ref{thm_menger}, the maximum number of edge-disjoint paths from $V(X)\setminus A$ to $v$ is equal $\cardinality{E(V(X)\setminus A, A)}$ which is even. This shows that $v$ is strongly even.
\end{proof}

However, in general, strongly even degree vertices need not be even.
\subsubsection*{Example} The right hand vertex in the graph-like continuum illustrated below is neither even nor odd but has strongly even degree.

\bigskip

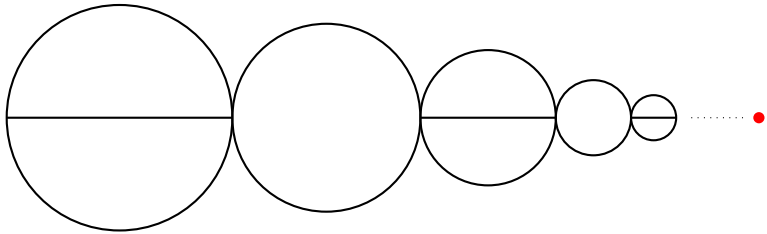
\begin{center}
\begin{tikzpicture}[thick]

\draw (1.5,0) circle[radius=1.5]; 
\draw (0,0) -- (3,0);

\draw (4.25,0) circle[radius=1.25]; 

\draw (6.4,0) circle[radius=0.9]; 
\draw (5.5,0) -- (7.3,0);

\draw (7.8,0) circle[radius=0.5]; 

\draw (8.6,0) circle[radius=0.3]; 
\draw (8.3,0) -- (8.9,0);

\draw[dotted,thin] (9.1,0) -- (9.8,0);

\draw (10,0) node[circle,fill=red, inner sep=1.5]  {};
\end{tikzpicture}
\end{center}
If each simple circle, 
\tikz{\draw circle[radius=0.2];}, in the above example is replaced with a copy of  \tikz{\draw circle[radius=0.2]; \draw ellipse[x radius=0.2, y radius=0.075];},
then in the resulting graph-like continuum the right hand vertex has strongly odd degree.

\medskip 

Our aim is to prove the following theorem, generalizing corresponding results of Bruhn \& Stein \cite{euler} and Berger \& Bruhn \cite{standard} for Freudenthal compactifications of 
graphs, and their Eulerian subspaces. 
Observe that this theorem can be rephrased as saying that although not every vertex of strongly even degree must be even, if \emph{all} vertices of a graph-like continuum have strongly even degree then they are  \emph{all}  even.

\begin{theorem}\label{bs_e}
A graph-like continuum is closed Eulerian if and only if all its vertices have strongly even degree.
\end{theorem}

It is an interesting open problem, whether the same conclusion holds under the assumption that all vertices have weakly even degree. 
The forward implication of  Theorem~\ref{bs_e} follows from Lemma~\ref{lem_deg}, Lemma~\ref{lem_even} and Proposition~\ref{theorem16}. Theorem~\ref{se_imp_E} establishes the converse. The plan for the proof of Theorem~\ref{se_imp_E} is to establish the contrapositive: if $X$ is a graph-like continuum which is not closed Eulerian then it contains a vertex without strongly even degree (i.e. of weakly odd degree). Lemma~\ref{lem_oddend} shows how a certain sequence of regions leads to such a vertex. Now if $X$ is a graph-like continuum which is not closed Eulerian, then by Theorem~(B) (iii) $\implies$ (i), there must be an odd cut in $X$. This provides the starting point for the sequence needed to apply Lemma~\ref{lem_oddend}. Theorem~\ref{thm_contraction} then  provides the `Contraction Machine' required to create the remaining elements of the sequence.

\medskip

If $v$ and $w$ are distinct vertices in a graph-like continuum $X$, and they both have  strongly odd degree, then after connecting them with a new edge they will both have strongly even degree. Conversely if they both have strongly even degree, then after removing an edge connecting them, they will have strongly odd degree. 
Hence, as we deduced the open version of Theorem~(B) from the closed version, we now derive the following characterization of open Eulerian graph-like continua.
\begin{theorem}
A graph-like continuum is open  Eulerian if and only if it has exactly two strongly odd degree vertices, and the rest  have strongly even degree.
\end{theorem}

\subsection{The odd-end lemma}

For a clopen subset $U \subset V(X)$, consider the induced graph-like space $X[U]$. We say that a clopen subset $U \subset V(X)$ is a \emph{region} if $X[U]$ is connected. By $\partial U \subset E(X)$ we denote the set of edges between the separation $(U, V \setminus U)$. This set is finite for regions $U$. Let us call a region $U$ of $X$ a $k$-region if $\cardinality{\partial U}=k$, and an \emph{even} or an \emph{odd} region depending on whether $k$ is even or odd. 

The following lemma generalizes the corresponding lemma of Bruhn \& Stein for locally finite graphs, \cite[p.7f]{euler}, to graph-like continua.

\begin{lemma}
\label{lem_oddend}
Let $X$ be a graph-like continuum, and let $E(X) = \Set{e_0,e_1,\ldots}$ be an enumeration of its edges. Assume there exists a sequence of regions $U_0, U_1, \ldots$ of $X$ with the following properties:
\begin{enumerate}
	\item\label{odd} $\cardinality{\partial U_n}$ is odd for all $n \in \N$,
	\item\label{nested} $ U_n \supset U_{n+1}$,	
	\item\label{notsmaller} if $D$ is a region of $X$ with $U_{n} \supset D \supset U_{n+1}$ then $\cardinality{\partial U_{n}} \leq \cardinality{\partial D} $ for all $n \in \N$, and 
    \item\label{nobigedges} $e_n \notin E[U_{n+1}]$.
\end{enumerate}
Then $X$ has a vertex which has weakly odd degree.
\end{lemma}

\begin{proof}
Since $A=\bigcap_{n \in \N} X[U_n]$ is a nested intersection of continua by (\ref{nested}), it is  non-empty and connected. It follows from (\ref{nobigedges}) that $A \subset V(X)$, so $A =\singleton{v}$ for some vertex $v$, since $V(X)$ is totally disconnected. Furthermore, compactness implies that $\set{U_n}:{n \in \N}$ is a neighborhood base for $v$ in $V(X)$.

Property (\ref{notsmaller}) together with Theorem~\ref{thm_menger} shows that for all $U_n$ the maximal number of edge disjoint arcs from $V\setminus U_n$ to $v$ equals $\cardinality{\partial U_n}$, so is odd by $(\ref{odd})$. Since the $U_n$ form a neighborhood base, it follows that $v$ has weakly odd degree.
\end{proof}

\subsection{The contraction machine}

Suppose we have an odd region $U_0$. We want to construct a sequence as in Lemma~\ref{lem_oddend}. 
If we recursively choose an odd region $U_{n+1}$ of minimal $\cardinality{\partial U_{n+1}}$ amongst all odd regions contained in $U_{n}$, then (\ref{odd}) and (\ref{nested}) are fine, and property (\ref{notsmaller}) is satisfied at least for all odd regions $D$ nested between $U_{n}$ and $U_{n+1}$. Following Bruhn \& Stein's idea \cite{euler}, our plan for evading all even regions $D$ with $\cardinality{\partial D}< \cardinality{\partial U_n}$ nested between $U_{n}$ and $U_{n+1}$ is roughly as follows: first, we contract all even regions $D \subset U_n$ with boundary smaller than $\cardinality{\partial U_n}$ to single points. Only then do we pick our region $U_{n+1}$. After uncontracting, this means that every small even region lies either behind $U_{n+1}$, or is completely disjoint from $U_{n+1}$. 

The next result formalizes this idea for contracting regions.

\begin{theorem}[Contraction Theorem]
\label{thm_contraction}
Let $X$ be a graph-like continuum such that all isolated vertices are even. Suppose further that $U \subset X$ is an odd region of $X$ such that for some even $m>0$, there is no infinite $k$-region with $k < m$ of $X$ contained in $U$. 

Then there is a collection $\script{M}$ of disjoint regions of $U$ such that after contracting every element of $\script{M}$ to a single point, the graph-like continuum $X/\script{M}$, with associated (monotone) quotient map $\pi \colon X \to X/\script{M}$, has the property that 
\begin{enumerate}[label=(\roman*)]
\item all isolated vertices of $X/\script{M}$ are even,
\item there are no infinite $k$ regions with $k \leq m$ contained in the region $\pi(U) \subset X/\script{M}$, and
\item if $D \subset U$ is an $\ell$-region of $X$, then there is an $\leq \ell$-region $D' \subset \pi(U)$ such that $\cardinality{\pi(D) \setminus D'} < \infty$.
\end{enumerate}
\end{theorem}

We divide the proof into a sequence of lemmas. For two subsets $A,B \subset X$, say that $A$ \emph{splits} $B$, or $B$ \emph{is split by} $A$, if $A\cap B \neq \emptyset \neq B \setminus A$.

\begin{lemma}
\label{lem_cleaninglemma}
Let $X$ be a graph-like continuum, and $U \subset X$ a region. Let $R, S_1, \ldots, S_n$ be infinite $m$-regions contained in $U$, where $S_1, \ldots, S_n$ are pairwise disjoint and $|R \setminus \bigcup_{i \leq n} S_i |= \infty$. 

If there is no infinite $k$-region with $k < m$ of $X$ contained in $U$, then there is an $m$-region $\tilde{R}$ which doesn't split any $S_i$ such that $|R \setminus ( \bigcup_{i \leq n} S_i \cup \tilde{R} )|< \infty$.
\end{lemma}

For the proof we need the following lemma, which can be proven, as for graphs, by a simple double-counting argument.

\begin{lemma}
\label{l_cutlemma}
Let $X$ be a graph-like space, and $Y,Z \subset V(X)$ clopen subsets. Then
$$\cardinality{\partial Y} + \cardinality{\partial Z} \geq \max \Set{\cardinality{\partial \p{Y\cap Z}}+\cardinality{\partial \p{Y \cup Z}}, \cardinality{\partial \p{Y\setminus Z}}+\cardinality{\partial \p{Z \setminus Y}}}.$$
\end{lemma}

\begin{proof}[Proof of Lemma~\ref{lem_cleaninglemma}]
Without loss of generality, assume that $S_1$ is split by $R$, i.e.\ that $R\cap S_1 \neq \emptyset \neq S_1 \setminus R$. We claim that one of $S_1 \cup R$ or $R \setminus S_1$ is an $m$-region. They are clearly clopen subsets of vertices of $X$.

Otherwise, since $S_1 \cup R$ and $R \setminus S_1$ are infinite, we have $\cardinality{\partial S_1 \cup R} > m$ and $\cardinality{\partial R \setminus S_1} > m$. Thus, Lemma~\ref{l_cutlemma} implies that $\cardinality{S_1 \setminus R} < m$ and $\cardinality{S_1\cap R} < m$, so both regions are are finite, contradicting that $S_1$ is infinite. 

Hence, one of $S_1 \cup R$ or $R \setminus S_1$ is has a boundary of size $m$, and they can't be disconnected, as otherwise their components had to be finite. Now put $R'$ to be either one of them, whichever was the $m$-region. Then $R'$ splits strictly fewer $S_i$ than $R$, but covers the same set together with the $S_i$.
Thus, we may pick $\tilde{R}$ to be such that it splits the fewest number of $S_i$, subject to the condition that $| R \setminus ( \bigcup_{i \leq n} S_i \cup \tilde{R})|< \infty$. By the preceding argument, it follows that $\tilde{R}$ does not split any of the $S_i$.
\end{proof}

Let $X$ be a graph-like continuum, and $U \subset X$ a region. Assume there is no infinite $k$-region with $k < m$ of $G$ contained in $U$. Let $\script{R}=\set{R_n}:{n \in \N}$ be an enumeration of all infinite $m$-regions of $G$ contained in $U$. Since each $R_i$ is faithfully represented by the finite cut $\partial R_i \subset E$, and $E$ is countable, there are indeed at most countably many such regions. Below we write $\script{S} \preceq \script{S}'$ if  $\script{S}$ is a refinement of $\script{S}'$, i.e. for all $S \in \script{S}$ there is $S' \in \script{S}'$ such that $S \subseteq S'$.

\begin{lemma}
\label{lem_contraction}
For every $n \in \N$ there are finite collections $\script{S}_n \subset \script{R}$ of disjoint $m$-regions of $U$ such that

(1) for all $R_j$ with $j \leq n$ we have $\cardinality{R_j \setminus \bigcup \script{S}_n} < \infty$, and
	(2) $\script{S}_n \preceq \script{S}_{n+1}$.
\end{lemma}

\begin{proof}[Construction]
We begin with $\script{S}_0 = \singleton{R_0}$. Suppose $\script{S}_n \subseteq \script{R}$ has been found satisfying the above properties. Applying Lemma~\ref{lem_cleaninglemma} with $R_{n+1}$ and the collection $\script{S}_n$, we obtain an infinite $m$-region $\tilde{R}_{n+1}$. We claim that 
$\script{S}_{n+1} =\{\tilde{R}_{n+1}\} \cup  \{S \in \script{S}_n : S\cap \tilde{R}_{n+1} = \emptyset\}$ is as desired.
Indeed, by construction, $\script{S}_{n+1}$ covers $R_{n+1}$ up to finitely many vertices; and $\bigcup \script{S}_{n} \subseteq \bigcup \script{S}_{n+1}$, so we preserved the covering properties of earlier stages. 
\end{proof}

We would like to contract the `maximal' $m$-regions (with respect to inclusion) contained in $\script{S} =\bigcup \script{S}_n$. However, for graph-like continua, there can be infinite non-trivial chains in $\script{S}$. Still, for any such chain $S_0 \subsetneq S_1 \subsetneq S_2 \subsetneq \cdots$ of $m$-regions, we can contract a suitable collection of disjoint even regions such that after contraction, all $S_n$ are finite. Our plan is to contract $S_{0}$, and each component of $S_{n+1} \setminus S_n$, to a single point for all $n \in \N$. Our next lemma provides the details for the second case. 

\begin{lemma}
\label{lem_componentsareeven}
Let $X$ be a graph-like continuum, and $U \subset X$ a region. Assume there is no infinite $k$-region with $k < m$ of $G$ contained in $U$.

If $S \subsetneq R$ are infinite $m$-regions contained in $U$, then $X[R \setminus S]$ has at most $m$ connected components, and every such component is an even region of $X$.
\end{lemma}

\begin{proof}
Note that since $X[R]$ is path-connected, it follows that every component of $X[R \setminus S]$ has to limit onto an end vertex of some $e \in \partial S$. Thus, $X[R \setminus S]$ has at most $\cardinality{\partial S} =m$ components. In particular, every component is clopen in $X[R \setminus S]$, and hence a region of $X$.

To see that $\partial \p{R \setminus S}$ is even, consider the graph induced by the multi-cut $(S, R \setminus S, V \setminus R)$. This graph has two even vertices, namely $\Set{S}$ and $\Set{V \setminus R}$. So  by the Handshaking Lemma, also the last vertex is even, i.e.\ $R \setminus S$ induces an even cut. Moreover, since in the contraction graph, both $\Set{S}$ and $\Set{V \setminus R}$ have degree $m$, it follows that the third vertex has the same number of edges to $\Set{S}$ and to $\Set{V \setminus R}$. In other words, we have
$\cardinality{\partial \p{R \setminus S} \cap \partial R} = \cardinality{\partial \p{R \setminus S} \cap \partial S}$.

Let $C$ denote the vertex set of one such component. It follows that in order to establish that $C$ is an even region, it suffices to show that 
\begin{align}\label{eq_1}
\cardinality{\partial C \cap \partial R} \geq \cardinality{\partial C \cap \partial S}.
\end{align}
Indeed, once we know that (\ref{eq_1}) holds for every component $C$, then $\cardinality{\partial R} = m = \cardinality{\partial S}$ gives equality in (\ref{eq_1}).
To see that (\ref{eq_1}) holds, note that if $\cardinality{\partial C \cap \partial R} < \cardinality{\partial C \cap \partial S}$, then we see that $\cardinality{\partial \p{S \cup C}} < m$, so this is a finite region, contradicting that $S$ was infinite.
\end{proof}

We now collapse all maximal $m$-regions in $\script{S} =\bigcup \script{S}_n$, and for every infinite proper chain in $\script{S}$ we perform the above contractions. Write $\script{M}$ for the disjoint collection of even regions we contract. Write $q_\script{M} \colon V(X) \to V(X/\script{M})$, which extends to a continuous (monotone) quotient map on $X \to X/\script{M}$ (where we also contract all potential loops), which we also call $q_\script{M}$. Note that since we contracted regions of a compact space, the map $q_\script{M} \colon X \to X/\script{M}$ is a closed, monotone map. In particular, this implies that preimages of regions are regions, see Theorem 9 of \cite{kuratowski}.

\begin{proof}[Proof of Theorem~\ref{thm_contraction}]
First, to see that $X/\script{M}$ is still a graph-like continuum, note that our countable family $\script{M}$ forms a null-sequence of clopen sets by Corollary~\ref{souslian}. It follows from the fact that \emph{if $X$ is separable metrizable, and $\script{A}=\set{A_n}:{n \in \N}$ a null-sequence of non-empty compact subsets of $X$, then $X/\script{A}$ is separable metrizable}, \cite[A.11.6]{mill}, that $X/\script{M}$ is a continuum. Further, it is graph-like, because its vertex set $V(X) / \script{M}$ is totally disconnected: If there was any non-trivial connected set $C \subset V(X/ \script{M})$, then $C$ cannot contain contracted vertices (they are isolated), so $C \subset V(X)$ is non-trivial connected, contradiction.

Item (i), that every isolated vertex of $X/\script{M}$ is even, follows from Lemma~\ref{lem_componentsareeven}, as we only contracted even regions. 

For (ii), that all $m$-regions of $X/\script{M}$ contained in $\pi(U)$ are finite, note that for any such m-region $D$ of $X/\script{M}$, the clopen vertex set $D'=\pi^{-1}(D)$ is an $m$-region of $X$. If $D'$ was infinite, then $D'$ appears in our list, so is covered by some finite $\script{S}_n$. Consider $S \in \script{S}_n$. Note that $S$ either gets contracted to a single point, or $S$ appears in an infinite chain with at most $n$ predecessors, in which case we contract $S$ to at most $(m\cdot n+1)$-many points. It follows that $D'$ gets contracted to finitely many points, i.e.\ $D$ is finite. 

For (iii), let $D$ be an $\ell$-region of $X$. There are at most $\ell$ many elements $M_1, \ldots, M_\ell \in \script{M}$ such that $\partial D \cap E[M_i] \neq \emptyset$. Now if $D \subset M_i$ for some $i$ then it is clear that $\pi(D)$ is finite. Otherwise, choose disjoint $m$-regions  $S_i \supset M_i$ in $\script{S}$. We claim that either $\tilde{D}= D\cup S_1$ or $\tilde{D} = D\setminus S_1$ is an $\leq \ell$-region. Otherwise, it follows from Lemma~\ref{l_cutlemma} that $\cardinality{\partial \p{D\cap S_1}}< m$ and $\cardinality{\partial \p{S_1 \setminus D}}<m$. So $S_1$ is finite, a contradiction. Continue with the other $S_i$. This gives us an $\leq \ell$-region $D'$, which differs from $D$ by finitely many $S \in \script{S}$.
\end{proof}

\subsection{Chasing odd regions}

After having established Theorem~\ref{thm_contraction}, the proof now proceeds essentially as in \cite{euler}. We need one more simple lemma.

\begin{lemma}
\label{nooddfiniteregions}
A graph-like continuum in which all isolated vertices are even does not contain finite odd regions.
\end{lemma} 
\begin{proof}
If $A=\Set{v_1, \ldots, v_n}\subset V(X)$ is a finite region, consider the finite graph induced by the multi-cut $(V \setminus A, \singleton{v_1}, \ldots, \singleton{v_n})$. Since all vertices $v_i$ are even, it follows from the Handshaking Lemma that also $\Set{V \setminus A}$ must be even.
\end{proof}

\begin{theorem}\label{se_imp_E}
A graph-like continuum is Eulerian if  all its vertices have strongly even degree.
\end{theorem}

\begin{proof}
Assume $X$ is not Eulerian. To prove the contrapositive we show $X$ contains a vertex without strongly even degree. If some isolated vertex does not have (strongly) even degree then we are done.  So assume all isolated vertices of $X$ are even. We  construct a sequence of graph-like continua $X=X_0, X_1, \ldots$ such that
\begin{enumerate}[label=(\alph*)]
	\item\label{aaa} $X_0 \xrightarrow{\pi_1} X_1 \xrightarrow{\pi_2} X_2 \xrightarrow{\pi_3} \cdots$ are successive quotients with monotone open quotient maps $\pi_n$, and write $f_n = \pi_n \circ \pi_{n-1} \circ \cdots \circ \pi_1$,
    \item\label{bbb} all $X_n$ have the property that all isolated vertices are even,
    \item\label{ccc} there are regions $V_n \subset X_n$ such that 
    	\begin{enumerate}[label=$(\arabic{*})'$]
			\item\label{odd'} $\cardinality{\partial V_n}$ is odd for all $n \in \N$,
			\item\label{nested'} $\pi_{n+1} (V_n) \supset V_{n+1}$,	
			\item\label{notsmaller'} any $\ell$-region of $X$ gets contracted to a $\leq \ell$-region of $ X_n$ modulo finitely many isolated vertices; and any $k$-region of $X_{n}$ contained in $V_n$ with $k < \cardinality{\partial V_n}$ gets contracted to finitely many vertices in $X_{n+1}$,
            \item\label{nobigedges'} $e_n \notin E[V_{n+1}]$.
		\end{enumerate}	
\end{enumerate}
Before describing the construction, let us see that that
$U_n = f_n^{-1}(V_n)$ 
defines regions satisfying the requirements of Lemma~\ref{lem_oddend}, and so $X$ has a vertex which does not have strongly even degree, as desired.

Indeed, as inverse images under monotone closed maps, they are connected, and hence regions in $X$. Next, it is easy to check that \ref{odd'} $\Rightarrow$ (\ref{odd}), \ref{nested'} $\Rightarrow$ (\ref{nested}) and \ref{nobigedges'} $\Rightarrow$ (\ref{nobigedges}). Finally, to see (\ref{notsmaller}), i.e.\ that $U_{n+1}$ does not lie behind some region $D$ of $U_n$ with small $\cardinality{\partial D}$, note that by \ref{notsmaller'}, this region $D$ would have been contracted to finitely many points in $X_{n+1}$, and hence $V_{n+1}$ would be finite, which is a contradiction by \ref{bbb} and Lemma~\ref{nooddfiniteregions}.

Now towards the construction of our sequence $X_0,X_1, \ldots$ with \ref{aaa}--\ref{ccc}. First, since $X=X_0$ is not Eulerian, it has an odd cut. By choosing $V_0=U_0$ to be an odd region of $X$ such that $\cardinality{\partial U_0}$ is minimal, we see that $V_0$ is as desired. Now suppose we have constructed $V_n \subset X_n$ according to \ref{aaa}--\ref{ccc}. Put $m_{n+1} = \cardinality{\partial V_n} -1$.

Recursively, apply Theorem~\ref{thm_contraction} with graph-like continuum $X^{(k)}$ and region $q_k\circ \cdots \circ q_1(U_n)$ to obtain graph-like continua 
$X_n=X^{(m_n)} \succ X^{(m_{n} + 1)} \succ \cdots \succ X^{(m_{n+1})}=X_{n+1}$ 
with corresponding monotone quotient maps
$q_k \colon X^{(k-1)} \to X^{(k)}$ for all even $0<k\leq m$.
Define
$\pi_{n+1} = q_{m_{n+1}} \circ \cdots \circ q_{m_{n} + 1} \colon X_n \to X_{n+1}$.

Note that Theorem~\ref{thm_contraction}(i) implies \ref{bbb}, and (ii) and (iii) imply \ref{ccc}\ref{notsmaller'}.  We now want to find an odd cut $V \subset \pi_{n+1}(V_n)$ such that $e_n \notin E(V)$. Towards this, note that $f_{n+1}(e_n)$ is either an isolated vertex $v$ of $X_{n+1}$, or $f_{n+1}(e_n)$ is an edge with end vertices say $x$ and $y$ in $X_{n+1}$. Find a multi-cut $\script{V}$ of $\pi_{n+1}(V_n)$ into regions which either displays $v$ as a singleton, or contains $x$ and $y$ in different partition elements. By Lemma~\ref{prop:sequenceoddcuts}, there is an odd region $V \in \script{V}$. Since isolated vertices of $X_{n+1}$ are even, $V$ is not the singleton $\singleton{v}$. In the other situation, note that in the induced graph $G(\script{V})$, the edge $f_{n+1}(e_n)$ is displayed as cross edge. In either case, we have  $e_n \notin E(V)$. 

Finally, amongst all odd regions of $X_n$ contained in $V$ pick any odd region $V_{n+1} \subset V$ such that $\cardinality{\partial V_{n+1}}$ is minimal. This choice satisfies items \ref{odd'}, \ref{nested'} and~\ref{nobigedges'}.
\end{proof}

\end{document}